\documentclass[11pt]{article}
\usepackage{amsmath}
\usepackage{amssymb}
\usepackage{amsthm}
\usepackage{anysize}
\usepackage{natbib}
\usepackage{xspace}

\usepackage{booktabs}
\usepackage{pdflscape}
\usepackage{dsfont}
\usepackage{enumerate}
\usepackage{graphicx}
\usepackage{floatrow}
\usepackage{bm}
\usepackage{multicol}
\usepackage{multirow}
\usepackage[dvips]{epsfig}

\usepackage{float}
\usepackage{color}
\usepackage[latin1]{inputenc}
\usepackage[english]{babel}
\usepackage{xcolor}

\usepackage{bbm}
\usepackage{mathtools}

\numberwithin{equation}{section}

\theoremstyle{plain}
\newtheorem{thm}{Theorem}

\newtheorem{lem}{Lemma}
\newtheorem{prop}{Proposition}
\theoremstyle{definition}
\newtheorem{defn}[thm]{Definition}
\theoremstyle{definition}

\theoremstyle{remark}
\newtheorem{oss}{Remark}
\newtheorem{esmp}{Example}

\marginsize{25mm}{25mm}{25mm}{25mm}

\newcommand{\I}{\mathds{1}}

\newcommand{\R}{\mathbb R}

\newcommand \balpha{{\boldsymbol \alpha}}
\newcommand \blambda{{\boldsymbol \lambda}}
\newcommand \btheta{{\boldsymbol \theta}}
\newcommand \bgamma{{\boldsymbol \gamma}}
\newcommand \bdelta{{\boldsymbol \delta}}

\renewcommand\S{{{\upshape S}$_{\balpha}(  \bdelta; \mu)$}\xspace}

\newcommand\Sb{{{\upshape S}$_{\alpha}(  \lambda, \beta; \mu)$}\xspace}

\newcommand \CTS{{{\upshape CTS}$_{\bgamma}( \btheta, \bdelta; \mu  )$} }

\newcommand \GTGS{{{\upshape GTGS}$_{\bgamma}(\balpha, \blambda, \btheta, \bdelta; \mu)$}\xspace}

\newcommand \GTGSs{{{\upshape GTGS}$^s_{\gamma}(\sigma, \alpha, \lambda, \theta  ; \mu)$}\xspace}

\newcommand \GTGSm{{{\upshape GTGS}$^-_{\gamma}(\alpha, \lambda, \theta, \delta; \mu  )$}\xspace}

\newcommand \GTGSp{{{\upshape GTGS}$^+_{\gamma}(\alpha, \lambda, \theta, \delta; \mu  )$}\xspace}

\newcommand \GTGSo{{{\upshape GTGS}$_\bgamma^0(\balpha, \blambda,  \bdelta; \mu  )$}\xspace}

\newcommand \GTGSos{{{\upshape GTGS}$^{0,s}_\gamma(\alpha, \lambda,  \delta; \mu  )$}  \xspace}

\newcommand \GTGSop{{{\upshape GTGS}$^{0,+}_\gamma(\alpha, \lambda,  \delta; \mu  )$}\xspace}

\newcommand \GTGSom{{{\upshape GTGS}$^{0,-}_\gamma(\alpha, \lambda,  \delta; \mu  )$}\xspace}

\newcommand \TGS{{{\upshape TGS}$_\balpha(\blambda, \btheta, \bdelta; \mu  )$}\xspace}

\newcommand \TGSs{{{\upshape TGS}$^s_\alpha(\lambda, \theta, \delta ; \mu )$}\xspace}

\newcommand \TGSm{{{\upshape TGS}$^-_\alpha ( \lambda, \theta, \delta; \mu  )$}\xspace}

\newcommand \TGSp{{{\upshape TGS}$^+_{\alpha}( \lambda, \theta, \delta; \mu  )$}\xspace}

\newcommand \TPL{{{\upshape TPL}$ (\alpha, \lambda, \theta, \delta)$}\xspace}

\newcommand \PL{{{\upshape PL}$ (\alpha, \lambda,  \delta)$}\xspace}

\newcommand \ML{{{\upshape ML}$ (\alpha, \lambda)$}\xspace}

\newcommand \TGSo{{{\upshape TGS}$^{0}_\balpha (\blambda,  \bdelta; \mu  )$}\xspace}

\newcommand \BL{{{\upshape BL}$(\balpha, \blambda,  \bdelta  )$}\xspace}

\newcommand \BG{{{\upshape BG}$(\blambda,\bdelta)$}\xspace}

\begin{document}

\title{Tempered geometric stable distributions and processes\footnote{
 Part of this work has been presented at the Third Italian Meeting on Probability and Mathematical Statistics, Bologna, June 2022 and at the FraCalMo Workshop, Bologna, October 2022. The author would like to thank Enrico Scalas, Lucio Barabesi, Luisa Beghin, Federico Polito and Piergiacomo Sabino for the helpful comments and discussions.}}
\author{Lorenzo Torricelli\footnote{University of Bologna, Department of Statistical Sciences ``P. Fortunati''. Email: lorenzo.torricelli2@unibo.it}  }
\date{\today}

\maketitle

\begin{abstract}
We introduce a notion of geometric tempering using exponentially-dampened Mittag-Leffler tempering functions and closely investigate the univariate case. Characteristic exponents and cumulants are calculated, as well as spectral densities. Absolute continuity relations are shown, and short and long time scaling limits of the associated L\'evy processes analyzed.
\end{abstract}


\section{Introduction}
  

In physics and natural sciences, systems evolving according to L\'evy stable distributions have long been observed (see \citet{chechkin2008introduction} for a survey). Modeling these phenomena as ``L\'evy flights/walks'', i.e., random processes with stationary independent stable increments has however the serious drawback of producing  dynamic probabilistic representations with infinite variance, which is problematic both  theoretically and in practice. In the seminal works of \citet{man+sta:95} and \citet{kop:95}, L\'evy a remedy was proposed by introducing a truncation procedure that, while inducing a minimal perturbation of the central part of distribution, impacted the tails in such a way as recovering finiteness of variance, and thus ultimately the Gaussian behavior of the process at large times. Since the observed convergence is very slow, for most practical purposes the L\'evy stable empirical paradigm is robust to this modification. This idea proved be very successful and enjoyed a vast range of applications even outside physics, most notably in economics and finance: see  \citet{stanley2003statistical} and the pioneering papers of \citet{boy+lev:00} and \citet{car+al:02} in option pricing.
In particular Koponen's idea of using a negative exponential cutoff function, yielding to a tractable analytic structure of the law, proved to be  very consequential. Such procedure began to be known under the name of ``tempering'' of stable distributions, and with this word it was originally meant  an exponential tilting of either the probability density function or the L\'evy measure.
 Based on the remark that a negative exponential is a simple example of a completely monotone function, 
 a general theory of tempering stable distributions using completely monotone functions is introduced by \citet{ros:07}, who shows that the properties of tempered laws can be fully described  using characteristic measures (``spectral measures''). 
Specific parametric examples include the $p$-tempered stable law of \citet{gra:16}, the generalized tempered stable process of \citet{ros+sin:10},  the modified and KR tempered stable  laws in \citet{kim+al:08, kim+al:09}. Other possible choices of tempering functions and related distributions are explored in \citet{ter+woy:06}.

In this work we introduce a novel class of tempered distributions, using as a tempering function an exponentially dampened negative Mittag-Leffler function.  The Mittag-Leffler function is defined as
\begin{equation}
E_\alpha(z)=\sum_{k=0}^\infty\frac{z^{k}}{\Gamma(k+\alpha)}, \qquad z \in \mathbb C, \quad \alpha \in (0,1]
\end{equation}
where $\Gamma$ stands for the Euler's gamma function.  For $r>0$ we thus consider distributions whose L\'evy measures are given in polar coordinates, and with radial part of the form
\begin{equation}\label{eq:levmesintro}
\frac{e^{-\theta r}E_{\alpha}(-\lambda r^\alpha)}{r^{1+\gamma}} dr, \: \quad r,\lambda, \theta>0,   \, \alpha \in (0,1), \, \gamma \in (0,1).
\end{equation}
In the above we see combined a stable inverse-power component, a classic exponential tempering function, and a negative Mittag-Leffler factor. The latter is a type of generalized exponential function with rapid decay around zero and power law tails.

The genesis of the proposed distribution can be retraced in the concept of geometric stability, introduced in \citet{kle+al:84}, answering the problem of finding a class of random variables which are infinitely-divisible under random geometric  summation.  A geometric stable law $X$ can be defined as an infinitely divisible distribution whose characteristic exponent $\psi_X$ can be written as a log-transform of a stable characteristic exponent $\psi_Z$ i.e. 
\begin{equation}\label{eq:geomstabintro}
\psi_X(z)=\log \left(1+ \psi_Z(z) \right)
\end{equation}
for some stable distribution $Z$.
Theory and applications are developed by, among the others,   \citet{kle+al:84},  \citet{lin:98},  \citet{san+pil:14}, \citet{mit+rac:91}, \citet{koz+rac:99}, \citet{koz+pod:01}. Known probability distributions such as the Laplace distribution and  \citet{pil:90} Mittag-Leffler law are included in this class. Convolutions of geometric laws lead to the Linnik distribution  (\citet{lin:63}, \citet{pak:98}, \citet{chr+sch:01}) and the Erlang/gamma distribution as special cases.   
Tempered versions of the positive Linnik law (TPL) have been considered in \citet{bar+al:16b} \citet{bar+al:16a} 
 and \citet{tor+al:21}. These are obtained by replacing the stable law $Z$ in \eqref{eq:geomstabintro} with a positive exponentially tempered stable variable. This makes it possible to introduce a notion of tempered geometric stability on the full space, by assuming a TPL for the radial part and then combining it with a finite measure on the unit sphere. We call these laws tempered geometric stable (TGS) distributions. As it turns out the L\'evy measure of a TPL distribution constructed coincide with the mixed exponential/Mittag-Leffler in \eqref{eq:levmesintro} when $\gamma=0$. Reverting to \eqref{eq:levmesintro} in its full generality, i.e.  considering instead an exponent $\gamma \in (0,2)$ at the denominator,  determines what we call a generalized tempered geometric tempered stable law (GTGS), the general object of investigation of this study. 
 
With regards to applications, the main motivation to study geometric tempering comes from economics and finance. A large number of independent statistical studies around the turn of the century (e.g. \citet{lux1996stable}, \citet{longin1996asymptotic}, \citet{gopikrishnan1998inverse}, \citet{plerou1999scaling}, \citet{gopikrishnan1999scaling})  have reported strong evidence of power law decay (and scaling) of market returns which survive even at long lags, and with typical Pareto exponents  of about 3. Therefore, according to these estimates, the existence of the variance in financial data is not called into question, while existence of higher moments remains a more delicate issue. It would then seem appropriate to seek dynamic return models whose distributions is of stable type in its central part and exhibits heavy law tails at longer lags, albeit with finite variance, and hence retaining Gaussian limiting properties.  Bringing these features together is not easy.  As argued in  \citet{chechkin2008introduction} 
 and \citet{stanley2003statistical}, abrupt density truncation, or exponential tempering  as proposed respectively by \citet{man+sta:95} and \citet{kop:95}, determine too much of a sharp cutoff and do not maintain the power law decay  for all the orders of magnitude for which it is observed in the empirical data, typically ranging from milliseconds to several days. One possible solution is proposed in \citet{sokolov2004fractional}, where truncation is achieved by means of a model whose density function solves a Fokker-Planck equation of fractional order. In contrast, in the present work we show how to obtain heavy tails with finite variance at arbitrary lags by tempering stable laws (and the associated L\'evy processes) using as a truncation function the Mittag-Leffler function, which corresponds to the particular case $\theta=0$ in \eqref{eq:levmesintro}. As we shall see the purely probabilistic approach followed leads to tractable analytic formulae and transparent asymptotic relations.

In this work we  find the characteristic exponents of GTGS laws, their TGS subcases, as well as those of their purely Mittag-Leffler tempered subcases, and discuss their analytical properties. Such exponents involve two interesting special functions, namely Dotsenko's $_2R_1$ generalized hypergeometric  function (\citet{dot:91}, \citet{vir+al:01}), and Lerch's transcendent $\Phi$. 
 We determine cumulants and spectral densities, as well as short and long time scaling limiting behavior. In particular, in the pure Mittag-Leffler case we observe that the large parameter scaling can follow a classic Gaussian limit, but can also converge to stable process, depending on the stability and Mittag-Leffler parameters. 
  Moreover, we analyze the absolute continuity conditions of GTGS laws within their own class as well as with with respect to stable laws.

 We review the basic notions needed and fix the notation in Section 2. In Section 3  GTGS distributions are introduced. Their characteristic functions and cumulants are discussed in Section 4. In Section 5 we analyze  spectral measures, short and long time limits as well as absolute continuity properties.  S In Section 6 we conclude and discuss possible developments.

\section{Preliminaries}
We begin by establishing the notation and recalling some concepts and notions required throughout the paper.
\subsection{Infinitely divisible laws and L\'evy processes}

Throughout the paper we fix a filtered probability space $(\Omega, \mathcal F, (\mathcal F_t)_{t \geq 0}, P)$ to which all the processes we mention are adapted. For a $d$-dimensional random variable $X$ on such space we denote by $\Psi_X: \mathbb R^d \rightarrow \mathbb C$ its characteristic function
\begin{equation}
\Psi_X(z)=E[e^{i \langle z , X \rangle}].
\end{equation}
An infinitely-divisible (i.d.) random variable  is fully characterized  by its characteristic exponent $\psi_X(z)$ i.e, a function $\psi_X: D \subseteq \mathbb R^d \rightarrow \mathbb C$ such that
\begin{equation}
\Psi_X(z):=E[e^{i \langle z , X \rangle}]=e^{t \psi_X(z)}.
\end{equation}
The exponent $\psi_X$ can be written in terms of a characteristic triplet  $(\mu, \Sigma, \nu)$ with $\mu \in \mathbb R^d$, $\Sigma$ is a positive definite $d \times d$ matrix, $\nu$ is a measure on $\mathbb R^d$ such that $\nu(0)=0$, $\int_{\mathbb R^d} (| x|^2 \wedge 1) dx <\infty$, with $| \cdot|$ indicating the Euclidean norm, and 
\begin{equation}\label{eq:LK}
\psi_X(z)=  i  \langle z,  \mu  \rangle- \langle z, \Sigma z \rangle/2+\int_{\mathbb R^d} \left( e^{i \langle z, x \rangle} -1 -   i \langle   z, x\rangle \I_{\{| x |<1\}}\right)\nu( d x)
\end{equation}
where $\langle \cdot,\cdot \rangle$ is the standard inner product on $\mathbb R^d$. The measure $\nu$ is called the L\'evy measure. When $\nu$ is absolutely continuous its corresponding density is the L\'evy density. The function  $x \I_{\{| x |<1\}}$ ensures the convergence of the integral in 0, and is called a truncation function. Equation \eqref{eq:LK} goes under the name of L\'evy-Khintchine representation. When $d=1$ and $X$ is a  positively supported i.d. random variable then one can also use the Laplace exponent $\phi$, defined to be the complex valued function such that
\begin{equation}
L_X(z):=E[e^{-s X}]=e^{-\phi_X(s)}, \qquad  s>0
\end{equation}
with

\begin{equation}\label{eq:Bern}
\phi_X(s)=a + bs +\int_{\mathbb R_+}(1-e^{-s x}) \nu(dx)
\end{equation}
for some triplet $(a, b,\nu )$, with $a, b  \geq 0$ and a positively-supported L\'evy measure $\nu$ satisfying $\int_{\mathbb R_+}(x \wedge 1)\nu(dx) <\infty$. One generic function $f$ enjoying representation \eqref{eq:Bern} is called a Bernstein function. A Bernstein function $f \in C^\infty(\mathbb R_+)$ is equivalently characterized by the property of being  such that, for all $n \in \mathbb N_0$

\begin{equation}
(-1)^{n-1}f^{(n)}(x) \geq 0.
\end{equation}
indicating the $n$-th derivative with the parenthetical exponent.
A completely monotone function on $\mathbb R_+$ is a function of class $C^\infty$ such that $f'$ is a Bernstein function.  A classic reference for Bernstein functions is \citet{sch+von:12}.

A  real-valued r.v. $X$ i.d.  is said to be self-decomposable (s.d.) if  for all $\alpha \in (0,1)$ can be written in law as $X=\alpha X+R_\alpha$, for some r.v. $R_\alpha$  independent of $X$. Equivalently, a real valued s.d. distribution is characterized by the property that its L\'evy measure is absolutely continuous and its density $v(x)$ writes as $v(x)=k(x)/|x|$, with $k(x)$ a function increasing on $\R_-$ and decreasing on $\R_+$. The function $k$ is  called the canonical density of $X$.

\bigskip

A L\'evy process $X=(X_t)_{t \geq 0}$ on $\mathbb R^d$ is a stochastically continuous process with independent and stationary increments. For a L\'evy process $X$,  $X_t$ is i.d. for all $t>0$. Conversely, given an i.d. random variable $\mathcal X$ there exists a unique in law L\'evy process $X$ such that $X_1=\mathcal X$ (\citet{sat:99}, Theorem 7.10). The characteristic exponent $\psi_X$ of a L\'evy process $X$ is by definition $\psi_{X_1}$. With abuse of terminology we shall refer to a L\'evy process by the name of its unit time  law. 



\bigskip

\subsection{Stable and tempered stable laws and processes}

A stable  S$_\alpha(\lambda, \beta; \mu)$ r.v. on $\R$, with $\alpha \in (0,2]$, $\mu \in \mathbb R, \lambda>0$ is one such that the equality $\Psi_X(z)^\alpha=\Psi_X(z^{1/\alpha})e^{i z \mu}$ holds. If $\mu=0$ the r.v. is said to be strictly stable. The r.v. $X$ is i.d. with characteristic exponent 
\begin{equation}\label{eq:stabled1}
\psi_X(z) =\left\{ \begin{array}{cc}
  -\lambda| z |^\alpha \left( 1 - i \beta \mbox{sgn}(z) \tan (\pi \alpha/2)\right) + i z \mu  & \mbox{ for } \alpha \neq 1\\
 -\lambda| z |\left( 1+ i \beta \frac{2}{\pi}  \mbox{sgn}(z)\log| z | \right) + i z \mu  & \mbox{ for }\alpha=1
\end{array} \right.
\end{equation}
where $\beta \in [-1,1]$, $z \in \mathbb R$. For $\beta=0$ the distribution is symmetric, for $\beta=\pm 1$ totally skewed respectively to the right and left. 
Stable laws are i.d., and an $\alpha$-stable L\'evy process $X=(X_t)_{t \geq 0}$ is one for which $X_1$ is stable (equiv. $X_t$ is stable for all $t$). 

On $\R^d$, $d \geq 1$, it is instead easier to define a stable process by means of its L\'evy triplet $(\mu, 0, \nu)$ where for $B \in \mathcal B(\mathbb R^d\setminus \{ 0\})$ we have the following polar representation 
 \begin{equation}\label{eq:stablegen}
 \nu(B)=\int_{S^{d-1}} \sigma(du) \int_{\mathbb R_+} \I_{B}(r u) \frac{dr}{r^{1+\alpha}}
 \end{equation}
with $S^d$ indicating the $d$-dimensional sphere. We denote the class of $\alpha$-stable distributions on $\mathbb R^d$ with S$_\alpha(\sigma, \mu)$. On $\mathbb R$, sometimes cases with different stability indices for the two distinct points in $S^0$ are taken into account, when the L\'evy measure $\nu$ has an absolutely continuous density $ v(x)$ of the form
\begin{equation}
v(x)=\frac{\delta_+}{x^{1+\alpha_+}}\I_{\{ x>0 \}}+\frac{\delta_-}{|x|^{1+\alpha_-}}\I_{\{ x <0 \}},
\end{equation}
with $\alpha_+,\alpha_- \in (0,2]$. When $\alpha_+=\alpha_-$ the constants $\delta_+$ and $\delta_-$ are related to the classical parametrization 	\eqref{eq:stabled1} by $\lambda=\delta_+ +\delta_-$ and $\beta=(\delta_+ -\delta_-)/(\delta_+ +\delta_-)$. In this case we denote the class of the stable distributions by S$_{\balpha}(\bdelta, \mu)$, with $\balpha=(\alpha_+, \alpha_-) \in (0,2] \times (0,2]$, $\bdelta=(\delta_+, \delta_-) \in \mathbb R^2_+$, $\mu \in \mathbb R$.  For a detailed survey on stable distribution see \citet{sat:99}, Chapter 3.

\bigskip


 A (classical) tempered stable distribution is an i.d. distribution with L\'evy measure $\nu$ given by 
\begin{equation}
\nu(dx)=\delta_+ \frac{e^{-\theta_+} x}{x^{1+\alpha_+}}\I_{\{ x>0 \}}+\delta_- \frac{e^{-\theta_-|x|}}{|x|^{1+\alpha_-}}\I_{\{ x <0 \}} \end{equation}
for $\alpha_+, \alpha_- \in (0,2]$, $\delta_+, \delta_- \geq 0$, see e.g. \citet{kop:95}, \citet{boy+lev:00}, \citet{car+al:02} for  applications, and \citet{kuc+tap:13} for a theoretical analysis. We shall use the notation  CTS$_{\balpha}( \btheta, \bdelta; \mu)$ with $\balpha=(\alpha_+, \alpha_-) \in \mathbb (0,2] \times (0,2]$,  $\btheta=(\theta_+, \theta_-) \in \mathbb R^2_+$, $\bdelta=(\delta_+, \delta_-) \in \mathbb R^2_+, \mu \in \mathbb R$. Special cases, other than the stable distributions when $\theta_\pm=0$, include the  bilateral Gamma BG$(\btheta,\bdelta)$  laws (e.g. \citet{kuc+tap:08})  is obtained for $\balpha=0$, and the ordinary positively-supported gamma G$(\theta, \delta)$ law, obtained for $\alpha_+=\alpha_-=\delta_-=0$.

\bigskip

A general approach to tempering stable laws is considered in \citet{ros:07}. Let $\sigma$ be a finite measure on $S^{d-1}$ and a function $q: \R_+ \times S^{d-1} \rightarrow \R^d$, $(r,u) \mapsto q(r,u)$, completely monotone in $r$ for all $u$. 
Then for $\alpha \in (0,2]$ an $\alpha$ tempered stable process on $\mathbb R^d$ is one with L\'evy triplet $(\mu,0, m)$ where $m$ is given in polar coordinates by
\begin{equation}\label{eq:rosrep}
m(dr, du)=\frac{q(r,u)}{r^{1+\alpha}}{\sigma(du)} dr.
\end{equation}
Since $q$ is completely monotone in the first variable, by the Bernstein Theorem (\citet{sch+von:12}, Theorem 1.4) there exists a family of probability measures $(Q_u)_{u \in S^{d-1}}$ supported on $\R_+$ for which 
\begin{equation}\label{eq:bernrep}
q(r,u)=\int_0^\infty e^{- r s}Q_u(ds).
\end{equation}
 Therefore for $B \in \mathcal B(\mathbb R^d \setminus \{ 0\})$ we can introduce two measures $Q$ and $R$ respectively by
\begin{equation}\label{eq:specmeas}
Q(B)=\int_{S^{d-1}} \sigma(du) \int_{\mathbb R_+} \I_{B}(r u)Q_u(dr) 
\end{equation}
and 
\begin{equation}\label{eq:rosmeas}
R(B)=\int_{\mathbb R^d} \I_{B}\left(\frac{x}{\|x\|^2}\right)\|x\|^\alpha Q(dx) .
\end{equation}
It turns out that $R$ and $Q$ are dual in the sense that
\begin{equation}
Q(B)=\int_{\mathbb R^d} \I_{B}\left(\frac{x}{\|x\|^2}\right)\|x\|^\alpha R(dx).
\end{equation}
We call  $Q$ the spectral measure and $R$ the Rosi\'nsky measure.
A given tempered $\alpha$ stable distribution is fully characterized by its  measures $Q$ or $R$, and  a drift vector $\xi \in \mathbb R^d$. 
 The measure $Q$ is useful for computer simulations (\citet{coh+ros:07}), whereas $R$ is pivotal to describe the analytical properties of the law.
 
  With abuse of notation, we denote the class of tempered stable distributions in the sense of Rosi\'nsky by TS$_\gamma(\sigma, q, \mu)$, TS$_\gamma(Q, \mu)$, or TS$_\gamma(R, \mu)$, which covers all  the different specifications illustrated above.

\subsection{Geometric stability and tempered geometric stability}\label{subsec:GS}

Geometric  stable  (GS) laws were introduced by \citet{kle+al:84} as a solution to the problem of characterizing i.d. laws whose infinite divisibility property holds true with respect to geometric summation. A geometric strictly stable law $X$ on $\mathbb R^d$ is one such that 

\begin{equation}\label{eq:geomtransf}
\psi_X(z)=\log \left(1+ \psi_Z(z) \right)
\end{equation}
is the characteristic exponent of some stable r.v. $Z$ on $\mathbb R^d$. 
In the case $d=1$ and when $\beta=0$ in \eqref{eq:stabled1}  we have the symmetric law with Laplace exponent
\begin{equation}
\phi_X(z)=\delta \log\left({1+ \lambda |z|^\alpha}\right), \qquad \delta, \lambda >0, \quad\alpha \in (0,2], \quad z>0.
\end{equation} The arising probability distribution is often called the Linnik distribution. When $\alpha=2$, $\delta=\lambda=1$ this reduces to the well-known Laplace distribution. 

Geometric stable processes and their applications have been studied in e.g. \citet{mit+rac:91},  \citet{vsi+al:06} \citet{koz+rac:99}, \citet{koz+sam:99},  \citet{kot+al:01}, \citet{koz+pod:01}. Positively-supported Linnik \PL laws can be seen as extension of the Mittag-Leffler \ML law introduced in \citet{pil:90}. 
We have, for the L\'evy measure of a \PL r.v. $X$ the L\'evy density
\begin{equation}\label{LevML}
v_X(x)=\frac{\delta} {x}E_\alpha\left(-\frac{x^\alpha}{\lambda}\right), \qquad x>0, \quad \alpha \in (0,1)
\end{equation}
and PL$(\alpha, \lambda, 1) \equiv $ \ML.

In \citet{bar+al:16a} a tempered version of the Linnik positive laws, denoted \TPL was investigated, and their associated processes later studied in \citet{kum+al:19a} and \citet{tor+al:21}. The resulting operation leads to the Laplace transform
\begin{equation}\label{eq:tilt}
\phi_{X}(s)= {\delta} \log \left({1+\lambda((\theta+s)^\alpha-\theta^\alpha)}\right)
\end{equation}
 Recalling that the Laplace exponent of a  CTS$^+_\alpha(\theta, \lambda; 0)$ law is $\lambda((\theta+s)^\alpha-\theta^\alpha)$, we see that		 \eqref{eq:tilt} is equivalent to requiring that \eqref{eq:geomtransf} now holds for a Laplace exponent of some \emph{tempered} stable positive law $Z$.  The expression for the L\'evy density is
\begin{equation}\label{eq:LevTL}
v_X(x)= \delta \frac{e^{-\theta x}}{x}E_\alpha\left(\frac{\lambda \theta^\alpha-1}{\lambda}x^\alpha\right), \qquad x>0, \alpha \in (0,1).
\end{equation}

TPL distributions seem to have first appeared in \citet{mer+al:11}, Example 5.7, to describe the  waiting times of a Poisson process subordinated to an inverse tempered stable subordinator (an increasing L\'evy process). A closed form expression in terms of the three-parameter Mittag-Leffler function is also available for the p.d.f.. See   \citet{bar+al:16a} and \citet{tor+al:21} for details.



\subsection{Special functions}

We denote with $E_{a, b}^c(z)$, $z  \in \mathbb C$, the \citet{pra:71} three-parameter Mittag-Leffler function given by \begin{equation}\label{eq:ML3}
E_{a, b}^c(z)=\sum_{k=0}^\infty \frac{(c)_k z^k}{k!\Gamma(ak+b)}, \qquad \text{Re}(a) >0, \quad \text{Re}(b)>0, \quad c, z \in \mathbb C,
\end{equation}
where $\Gamma(\cdot)$ is the Euler's Gamma function and $(c)_k=\Gamma(c+k)/\Gamma(c)$ the Pochhammer symbol. The standard and two-parameter Mittag-Leffler functions  $E_a$ and $E_{a, b}$ coincide with $E_{a,1}^1$ and $E_{a,b}^1$ respectively.  All these functions are entire. 

The following leading asymptotic order of the three-parameter Mittag-Leffler function (e.g. \citet{gar+gar:18}), as $x \in \mathbb R$, $x \rightarrow \infty$ will be useful
\begin{equation}\label{eq:MLasymptInf3pam}
E_{a,a c}^c(-x^a) \sim - c \frac{x^{-a(c+1) }}{\Gamma(-a )}, \qquad a \in (0,1), \quad c>0
\end{equation}
from which we can recover  the well-known formula (e.g. \citet{hau+al:11})
\begin{align}\label{eq:MLasymptInf}
E_\alpha(- x^\alpha) \sim &\frac{x^{-\alpha}}{\Gamma(1-\alpha) }, \quad x   \rightarrow \infty, \qquad \alpha \in (0,1). 
\end{align}
From the definition it is also clear that
\begin{align}  E_\alpha(- x^\alpha) \sim & 1 - \frac{x^\alpha}{\Gamma(1+ \alpha)} \sim \exp\left(- \frac{x^\alpha}{\Gamma(1+\alpha)} \right) , \quad x \rightarrow 0. \label{eq:MLasympt0}
\end{align}
See  \citet{hau+al:11} or \citet{gor+al:14} for a comprehensive introduction on Mittag-Leffler functions.
\bigskip

The $_2R_1$ generalized hypergeometric Wright-type function studied by \citet{dot:91} and \citet{vir+al:01}, is given by for $\tau >0, a,b,c \in \mathbb C, c \not \in  \mathbb Z_{\leq 0}$
\begin{equation}\label{eq:DotHG}
_2R_1(a,b,c, \tau; z)= \frac{\Gamma(c)}{\Gamma(b)}\sum_{k=0}^\infty \frac{(a)_k \Gamma(b+\tau k)}{\Gamma(c+ \tau k)} \frac{z^k}{k!}, \qquad \mbox{Re}(c-b-a)>0.
\end{equation}
Under this assumptions the power series expansion above converges absolutely for $|z| \leq 1$, but the $_2R_1$ function can be continued analytically on $\mathbb C \setminus(1,\infty)$. 
 For $|z|<1$, $|\text{Arg}(-z)|<\pi$  the expression of the continuation is given by e.g \citet{kil+sa:03}.   Notice that 
  $_2R_1(a,b,c, 1; z)=\, _2F_1(a,b,c; z)$, Gauss' hypergeometric function, so that in particular  $_2R_1(a,b,b, 1; z)=(1-z)^{-a}$  .
Furthermore, the generalized  hypergeometric $_2R_1$ can be represented in terms of the normalized Fox-Wright 
(\citet{wri:35}) function $_p \Psi^*_q$ as
\begin{equation}\label{eq:Dots}
_2R_1(a,b,c, \tau; z)=\, _2\Psi^*_1 \left[ \begin{array}{c c}  (a,1) & (b, \tau)  \\ (c, \tau) &  \end{array} ; z \right].
\end{equation}
See also \citet{bra:63} and \citet{kar+pri:19} for details on the analytic continuation of the $_p\Psi_q$ function.

The Lerch transcendent function $\Phi$ is defined as the convergent series
\begin{equation}\label{eq:Lerch}
\Phi(z, s, a)=\sum_{k=0}^\infty \frac{z^k}{(a+k)^s}, \qquad z, s \in \mathbb C, |z|<1,  a \notin \mathbb Z_{\leq 0}.   \end{equation}
It relates to the Polilogarithm Li$_s(z)$ through Li$_s(z)/z=\Phi(z,s,1)$  and extends the Hurwitz/Riemann Zeta functions, in that $\zeta(s,a)=\Phi(1,s,a)$, $\zeta(s)=\Phi(1,s,0)$.  Again, even though the series representation above is only valid for $|z| < 1$, and for $|z|=1$ if Re$(s)>0$, analytic continuations are available, and have been recently studied (\citet{fer+al:17}).

\section{Geometric tempered stable distributions}

 
In order to construct a multivariate geometric tempered stable distributions we follow the approach in \citet{ros:07} of defining a L\'evy measure in polar coordinates using a positive probability measure for the radial part and a  measure $\sigma$ to move it along the unit sphere.

 Let $d \geq 1$ and $\theta, \lambda: S^{d-1} \rightarrow \mathbb R_+$, $\alpha: S^{d-1} \rightarrow (0,1]$  continuous  functions.  Let the tempering function in \eqref{eq:rosrep} be given by

\begin{equation}\label{eq:TGS}
q(r,u;\alpha, \lambda, \theta )=e^{-\theta(u) r}E_{\alpha(u)}\left(-\lambda(u) r^{\alpha(u)}  \right).
\end{equation}
In order to avoid degeneracies, the following technical assumption is needed:
\begin{itemize}
\item[\textbf{(B)}] either $\theta$, or both $\alpha$ and $\lambda$, are bounded away from zero.
\end{itemize}
To begin with one needs to show that \eqref{eq:TGS} is a legitimate tempering function, and that this is so even in the case $\gamma=0$. 

\begin{prop}\label{prop:legit} Let $\sigma(du)$ be a finite measure on $S^{d-1}$. For $\gamma \in (0,2)$ the function 
$q(r,u;\alpha, \lambda, \theta )$ is a tempering function in the sense of {\upshape \citet{ros:07}}. Furthermore, under assumption {\upshape \textbf{(B)}}.  
\begin{equation}
m_0(dr,du; \sigma, \alpha, \lambda, \theta):=\frac{q(r,u;\alpha, \lambda, \theta )}{r}\sigma(du)dr
\end{equation} 
is a L\'evy measure. Therefore under the given conditions
\begin{equation}
m_\gamma(dr,du; \sigma, \alpha, \lambda, \theta):=
\frac{q(r,u;\alpha, \lambda, \theta )}{r^{1+\gamma}}\sigma(du)dr
\end{equation}
is a L\'evy measure for all $\gamma \in [0,2)$.
 \end{prop}
\begin{proof}

 The exponential function is completely monotone for all $\theta <0$, and so  is the function $E_\alpha(- \, \cdot)$ (\citet{pol:48}). Furthermore the function $x \rightarrow \lambda x^\alpha$, $\alpha \in (0,1)$ is positive with a completely monotone derivative, so that its composition with $E_\alpha(- \, \cdot)$ is also completely monotone. Therefore $q(\cdot, u)$ is completely monotone, it being a product of completely monotone functions (\citet{sch+von:12}, Corollary 1.6), for all possible values of $u$. Observe also that  $q(0+, u)=1$, for all $u$.
  
To complete the proof we only need to show that   $\int_{\mathbb R^d}((ru)^2 \wedge 1)m_0(dr, du) < \infty$, since for positive values of $\gamma$ this is a particular cases of \citet{ros:07} analysis. From $\exp(\cdot), E_\alpha(- \, \cdot) \leq 1 $ we have  \begin{equation}
\int_{S^{d-1}}\int_0 ^ 1 r^2  \frac{
q(r,u;\alpha, \lambda, \theta )}{r}dr \sigma(du) < \sigma(S^{d-1})\int_0^1 r  dr< \infty  
  \end{equation}
Also, by assumption \textbf{(B)} it follows that if $\theta_*= \min_{S^{d-1}}\theta(u)=0$ then both $\lambda_*= \min_{S^{d-1}}\lambda(u),$ and $ \alpha_*= \min_{S^{d-1}}\alpha(u)$ are strictly positive. Then in view of \eqref{eq:MLasymptInf} it follows
 \begin{align}
\int_{S^{d-1}}\int_1 ^ \infty   \frac{
q(r,u;\alpha, \lambda, \theta )}{r}dr \sigma(du) & <  \sigma(S^{d-1})\int_1^\infty r^{-1}  E_{\alpha_*}(-\lambda_* r^{\alpha_*} ) dr \nonumber \\ &\sim \frac{\sigma(S^{d-1})}{\lambda_* \Gamma(1-\alpha_*)} \int_1^\infty r^{-1-\alpha_*} dr  < \infty.
  \end{align}   
   Similarly if $\lambda_*, \alpha_*=0$ then $\theta_*>0$ whence
 \begin{equation}
\int_{S^{d-1}}\int_1 ^ \infty   \frac{
q(r,u;\alpha, \lambda, \theta )}{r}dr \sigma(du) < \sigma(S^{d-1})\int_1^\infty r^{-1} e^{-r \theta_*} dr < \infty.
  \end{equation}    
\end{proof}

We can thus define a general $d$-dimensional GTGS variable as follows.

\begin{defn} 
Let $d \geq 1$, $\gamma \in [0,2)$, $\sigma(du)$ be a finite measure on $S^{d-1}$, and $\mu \in \mathbb R^d$. A generalized tempered geometric  $\gamma$-stable distribution GTGS$_{\gamma}(\sigma, \alpha, \lambda, \theta; \mu  )$  on $\mathbb R^d$ is the i.d. distribution determined by the L\'evy  triplet $(\mu, 0, m_\gamma(dr, du; \sigma, \alpha, \lambda, \theta))$. A tempered geometric stable distribution  TGS$(\sigma, \alpha, \lambda, \theta; \mu  )$ is a GTGS$_{0}(\sigma, \alpha, \lambda, \theta; \mu  )$ distribution.
\end{defn}
When $\sigma(du)=\sigma du$, $\sigma>0$ and $\lambda, \theta, \alpha$ are constant functions, the GTGS distribution is  rotationally-invariant (i.e. symmetric).

In this paper we will explore in detail the case $d=1$, which is easier being $S^{0}=\{-1, 1\}$ and a great deal can be said on the analytic structure of the distributions. In such case the spherical measure  reduces to
\begin{equation}\label{eq:sigma0}
\sigma(du)=\delta_+\delta_1 (du)+ \delta_- \delta_{-1}(du)
\end{equation} for $\delta_+, \delta_- \geq 0$ and $\delta_x$ is the Dirac delta measure concentrated in $x \in \R$. We have the following expression for the L\'evy density, with $\alpha(u)=\alpha_\pm$, $\theta(u)=\theta_\pm$, $\lambda(u)=\lambda_\pm$ 
\begin{equation}
m_\gamma(dr,du; \sigma, \alpha, \lambda, \theta)=\delta_+ \frac{e^{-\theta_+ r}}{r^{1+\gamma}} 	  E_{\alpha_+} ( -\lambda_+ \, r^{\alpha_+} ) dr \delta_1(du) +  \delta_- \frac{e^{-\theta_- r}}{r^{1+\gamma}} E_{\alpha_-}(- \lambda_- \,r^{\alpha_-} )dr \delta_{-1}(du) . 
\end{equation}
One can easily extend this definition by assuming the stability indices for the negative and positive parts are not necessarily equal i.e. if $\gamma_+ \neq, \gamma_- \in [0,2]$, using the vector notation \begin{align}
\balpha&=(\alpha_+, \alpha_-) \in (0,1] \times (0,1], \nonumber \\ \blambda&=(\lambda_+, \lambda_-) \in (0,\infty) \times (0, \infty), \nonumber \\ \btheta&=(\theta_+, \theta_-) \in [0,\infty) \times [0, \infty),  \\ \bdelta&=(\delta_+, \delta_-) \in [0,\infty) \times [0, \infty) \setminus \{(0,0)\}, \nonumber \\ \bgamma&=(\gamma_+, \gamma_-), \in [0,2) \times [0,2) \nonumber
\end{align} in Cartesian coordinates we have the expression for the L\'evy density
\begin{equation}\label{eq:onedimlev}
m_{\bgamma}(x; \balpha, \blambda, \btheta,  \bdelta )=\delta_+ \frac{e^{-\theta_+ x}}{x^{1+{\gamma_+}}} 	  E_{\alpha_+} ( -\lambda_+ \, x^{\alpha_+} ) \I_{\{x>0\}} +  \delta_- \frac{e^{-\theta_- |x|}}{|x|^{1+\gamma_-}} E_{\alpha_-}(- \lambda_- \,| x|^{\alpha_-} )\I_{\{ x<0\}}. 
\end{equation}
When the L\'evy measure is symmetric, or it has positive/negative support  we write respectively $m_{\gamma}(x; \alpha, \lambda, \theta,  \delta )$, $m^+_{\gamma}(x; \alpha, \lambda, \theta,  \delta )$, $m^-_{\gamma}(x; \alpha, \lambda, \theta,  \delta )$, where the constants denote the only surviving value for each parameter vector. In these two latter instances the corresponding GTGS laws are said to be spectrally positive (resp. negative). It is important to recall the spectral positive/negativity is not the same as the corresponding probability laws being positively/negatively supported, although this equivalence holds true when $\gamma \in [0,1)$ and $\mu= \int_{\{|x|<1\}}m^\pm_{\gamma}(x; \alpha, \lambda, \theta,  \delta )$ (see \citet{sat:99}, Theorem 24.10).  
Again $\int_{\mathbb R}(x^2 \wedge 1)m_\bgamma(x; \balpha, \blambda, \btheta,  \bdelta) dx < \infty$,  so  we can introduce the  main definition.
\begin{defn}
Let $\mu \in \mathbb R$ and $m_{\bgamma}(x; \balpha, \blambda, \btheta,  \bdelta )$ be given by \eqref{eq:onedimlev}.   A generalized tempered  geometric  $\bgamma$-stable distribution \GTGS on the real line  is an i.d. distribution whose L\'evy triplet is given by  $(\mu, 0, m_\bgamma(x; \balpha, \blambda, \btheta, \bdelta ) dx)$. A tempered geometric stable \TGS  distribution is a   {\upshape GTGS}$_{0}(\balpha, \blambda, \btheta, \bdelta; \mu  )$ distribution. Symmetric GTGS and TGS distributions are denoted by  \GTGSs and \TGSs; spectrally positive (resp. negative)  GTGS and TGS distributions are denoted by \GTGSp and \TGSp (resp. \GTGSm and \TGSm). A GTGS$_\bgamma(\balpha, \blambda, 0, \bdelta; \mu)$ distribution  is indicated by \GTGSo (resp. \GTGSos, \GTGSop, and \GTGSom).
\end{defn}

\begin{oss} Notice that  for GTGS distributions the ``stability parameter'' is declared to be the  exponent $\gamma$ in the \citet{ros:07} framework, whereas for TGS distributions coincides with the Mittag-Leffler parameter. As we will clarify further on, the reason is that TGS random variables are  extensions of positive geometric stable laws, and for the latter the stability parameter is that of the underlying stable distributions.
\end{oss}

\begin{oss}\label{oss:TGTSselfdec} It is clear from \eqref{eq:onedimlev} that distributions in the \GTGS class are self-decomposable with canonical density 
\begin{equation}\label{eq:GTGSk}
k_{\bgamma}(x; \balpha, \blambda, \btheta,  \bdelta )=\delta_+ \frac{e^{-\theta_+ x}}{x^{\gamma_+}} 	  E_{\alpha_+} ( -\lambda_+ \, x^{\alpha_+} ) \I_{\{x>0\}} +  \delta_- \frac{e^{-\theta_- |x|}}{|x|^{\gamma_-}} E_{\alpha_-}(- \lambda_- \,| x|^{\alpha_-} )\I_{\{ x<0\}}. 
\end{equation}
\end{oss}

\begin{oss}\label{special}
 With the numerical constants denoting the corresponding constant functions,  we have the following particular cases:
\begin{itemize}
\item[(i)] Let  $\alpha_n \rightarrow 0$ be any sequence of real numbers. Then $X_n \sim$  GTGS$_\gamma(\sigma, \alpha_n, \lambda, \theta, \mu  )$ is such that $X_n \rightarrow ^d X$ with $X \sim \mbox{TS}_\gamma(\exp(- r \theta), \sigma^{\lambda} )$, where $\sigma^\lambda(du)=\frac{\sigma(du)}{1 +\lambda(u)}$, provided that $|\lambda(u)|<1$ for all $u$;
\item[(ii)]  GTGS$_\gamma(\sigma, 1, \lambda, \theta,\mu  )=\mbox{TS}_\gamma(\exp(- r (\theta+\lambda)), \sigma, \mu  )$;
\item[(iii)]  GTGS$_\gamma(\sigma, 1, \lambda ,0, \mu )=\mbox{TS}_\gamma(\exp(- r \lambda), \sigma , \mu )$;
\item[(iv)]If $\lambda_n \rightarrow 0$ is a real-valued sequence, and $\gamma>0$, then  $X_n \sim$ GTGS$_\gamma(\sigma, \alpha, \lambda_n , \theta )$ is such that $X_n \rightarrow^d X$, with $X \sim \mbox{TS}_\gamma(\exp(- r \theta), \sigma, \mu  )$;
\item[(v)] If $\lambda_n \rightarrow 0$ is a real-valued sequence, and $\gamma>0$, then  $X_n \sim$ GTGS$_\gamma(\sigma, \alpha, \lambda_n , 0 )$ is such that $X_n \rightarrow^d X$, with $X \sim$ S$_\gamma(\sigma, \mu)$;
\item[(vi)]  TGS$_\alpha^+(\lambda, \theta, \delta; \mu_0)=$ TPL$\left(\alpha, \frac{1}{\lambda +\theta^{\alpha}}, \theta, \frac{  \delta}{\alpha}\right)$ where $\mu_0=\int_{\{x<1\}}m_0^+(x; \alpha, \lambda, \theta,  \delta)dx.$
\end{itemize}
Furthermore we have (up to a possible location shift) the various special classes within the GTGS class on $\mathbb R$:
\begin{itemize}
\item[(vii)] GTGS$_{\bgamma}((1,1), \blambda, \btheta,  (\delta, \delta); \mu)$  are distributions in the general \citet{kop:95} and KoBol, \citet{boy+lev:00}, classes;
\item[(viii)] GTGS$_{(\gamma,\gamma)}((1,1), \blambda, \btheta,  (\delta, \delta); \mu)$  are CGMY distributions, \citet{car+al:02};
\item[(ix)] GTGS$_{(0,0)}((1,1), \blambda, \btheta,  \bdelta;\mu)=$ BG$(\blambda+\btheta,\bdelta)$, \citet{kuc+tap:08};
\item[(x)] GTGS$_{(\gamma,\gamma)}((1,1), \blambda, (\theta, \theta),  \bdelta; \mu)$ are the i.d. innovations of truncated L\'evy flights, \citet{man+sta:95}.
\end{itemize}
\end{oss}
\begin{proof}
Using continuity of the Mittag-Leffler function, $E_1(-\lambda(u) x)=\exp(-\lambda(u) x)$, $x \in \R$, $E_\alpha(0)=1$ for all $\alpha \in (0,1]$ and $E_\alpha(-\lambda(u)) \rightarrow 1/(1+\lambda(u))$ as $\alpha \rightarrow 0$ (\citet{hau+al:11}) together with dominated convergence as needed, (i)-(vi) follow operating on the L\'evy measures. For (vi) write 
\begin{equation}
m^+(x; \alpha, \lambda, \theta,  \delta)
= \delta \frac{e^{-\theta x}}{x} 	  E_{\alpha} ( -\lambda \, x^{\alpha} )= \alpha c_1 \frac{e^{-\theta x}}{x} 	  E_{\alpha} \left(\frac{1-c_2\theta^{\alpha}}{c_2} \, x^{\alpha} \right)  
 \end{equation}
with $c_1=\delta /\alpha$, $c_2=1/(\lambda +\theta^{\alpha})$, which matches the expression of the TPL L\'evy measure in Proposition 2.1 of \citet{tor+al:21}, which is given there for the Laplace exponent. Since $\phi_X(-iz)=\psi_X(z)+ iz \mu_0$ the result follow. Finally (vii)-(x) are simple parameter specifications of the prior cases.
\end{proof}

Remark \ref{special} clarifies that GTGS distributions specialize to both CTS and TPL distributions. As we shall argue in the following, the intermediate cases between these boundary distributions can be considered as ``tempered'' versions of geometric stable laws.


\section{Characteristic exponents}

We begin a theory of one-dimensional GTGS distributions by the determination of their characteristic exponents. We have a  divide depending on whether $\theta_+, \theta_->0$ or $\theta_+=\theta_-=0$. The first case can be computed directly. The second corresponds to a tempering function which is of pure Mittag-Leffler type and requires using  a limiting argument on the analytic continuation of the positive $\theta$ case.  The mixed cases $\theta_+>0, \theta_-=0$, $\theta_+=0, \theta_->0$ are of course possible and can be obtained by combining negative and positive parts as necessary, and thus will not be considered. Cumulants are also discussed at the end of the section.

\begin{thm}\label{thm:CF} Let $X$ be a \GTGS r.v. with $\theta_+, \theta_->0$ and $\alpha_+, \alpha_- \in (0,1)$, let $m$ be the L\'evy density of $X$ and set $\mu_0=\int_{\{|x|<1\}}x m(x)dx$
and $\mu_1=\int_{\{|x|>1\}}x m(x)dx$. Then

\begin{itemize}
\item[(i)] if $\gamma_-=\gamma_+=0$ then
 \begin{align}\label{eq:CFTGSg0}
\psi_X(z)=\frac{\delta_+}{\alpha_+} \log \left(\frac{\theta_+^{\alpha_+} +\lambda_+}{(\theta_+-i z)^{\alpha_+}+\lambda_+} \right)+ \frac{\delta_-}{\alpha_-} \log \left(\frac{\theta_-^{\alpha_-} +\lambda_-}{(\theta_+ + i z)^{\alpha_-}+\lambda_-} \right) + i z ( \mu- \mu_0);
\end{align}
\item[(ii)]  
if $\gamma_+, \gamma_- \in (0,2)$, $\gamma_\pm \neq 1$, then \begin{align}\label{eq:CFTGSgl2}\psi_X(z)=&\Gamma(-\gamma_+) \delta_+ \left( (\theta_+ - i z )^{\gamma_+} \, _2R_1\left(1, -\gamma_+, 1, \alpha_+;\frac{-\lambda_+}{ (\theta_+-i z)^{\alpha_+}}\right)  \right. \nonumber  \\ & \left. + i z \, \gamma_+ \theta_+^{\gamma_+-1} \,  _2R_1\left(1, 1-\gamma_+, 1, \alpha_+;-\frac{\lambda_+}{ \theta_+^{\alpha_+}} \right) -\theta_+^{\gamma_+} \,  _2R_1\left(1, -\gamma_+, 1, \alpha_+;-\frac{\lambda_+}{ \theta_+^{\alpha_+}} \right) \right)+ \nonumber \\ &\Gamma(-\gamma_-) \delta_- \left( (\theta_- + iz )^{\gamma_-} \, _2R_1\left(1, -\gamma_-, 1, \alpha_-;\frac{-\lambda_-}{ (\theta_- + i z)^{\alpha_-}}\right)  \right. \nonumber  \\ & \left. - i z \gamma_- \, \theta_-^{\gamma_--1} \,  _2R_1\left(1, 1-\gamma_-, 1, \alpha_-;-\frac{\lambda_-}{ \theta_-^{\alpha_-}} \right) -\theta_-^{\gamma_-} \,  _2R_1\left(1, -\gamma_-, 1, \alpha_-;-\frac{\lambda_-}{ \theta_-^{\alpha_-}} \right) \right) \nonumber \\ &\phantom{xxxxxxxxxxxxxxxxxxxxxxxxxxxxxxxxxxxxxxxx}+  i z (\mu_1 +\mu );
   \end{align}
\item[(iii)] 
if $\gamma_+=\gamma_-=1$ then \begin{align}\label{eq:CFTGS1}
\psi_X(z)=& \delta_+ \frac{\lambda_+}{\alpha_+} \left(   \frac{\Phi\left(-\frac{\lambda_+}{\theta_+^{\alpha_+}}, 1, \frac{\alpha_+-1}{\alpha_+}\right) }{\theta^{\alpha_+-1}}-  \frac{\Phi\left(-\frac{\lambda_+}{(\theta_+-iz)^{\alpha_+}}, 1, \frac{\alpha_+-1}{\alpha_+}\right) }{(\theta_+-iz)^{\alpha_+-1}}  \right) \nonumber \\ & \phantom{xxxxxxxxxxxxxxxxx} + \delta_+\frac{\theta_+-i z}{\alpha_+} \log \left(\frac{(\theta_+-i z)^{\alpha_+}+ \lambda_+}{\theta_+^{\alpha_+}+ \lambda_+ }\right) + \nonumber \\ & \delta_- \frac{\lambda_-}{\alpha_-} \left(   \frac{\Phi\left(-\frac{\lambda_-}{\theta_-^{\alpha_-}}, 1, \frac{\alpha_--1}{\alpha_-}\right) }{\theta^{\alpha_--1}} -\frac{\Phi\left(-\frac{\lambda_-}{(\theta_-+iz)^{\alpha_-}}, 1, \frac{\alpha_--1}{\alpha_-}\right) }{(\theta_-+iz)^{\alpha_--1}}   \right) \nonumber \\ & \phantom{xxxxxxxxx} + \delta_- \frac{\theta_- + i z}{\alpha_-} \log \left(\frac{(\theta_- + i z)^{\alpha_-}+ \lambda_-}{\theta_-^{\alpha_-}+ \lambda_- }\right)+ i z(\delta_+-\delta_-+\mu_1+\mu). 
\end{align}
\end{itemize}
The remaining cases can be derived from the given expressions by combining positive and negative parts as needed. 
Furthermore, all the above characteristic functions can be analytically continued  on   $S= \{z \in \mathbb C:  \mbox{\upshape{Im}}(z) \in (-\theta_-,  \theta_+) \}$.
\end{thm}

\begin{proof}
We only treat the positive part, the negative one being identical with the obvious parameter modification, and by substituting $x$ with $|x|$. We thus remove the subscripts 
for ease of notation.

Assume $\gamma=0$. We notice that in that case $x \, m(x)$ is integrable around zero, and we can compute the L\'evy-Khintchine integral without truncation. Interchanging series and integral using  Fubini's Theorem, we have 
\begin{align}\label{eq:CEgmain}
 \int_0^\infty & \left(e^{i z  x}-1   \right) q(x)x^{-1-\gamma} dx= \int_0^\infty \sum_{k=1}^\infty \frac{(i z x)^k}{k !}  \frac{e^{-\theta x}}{x} E_{\alpha} \left( -\lambda \, x^{\alpha} \right) dx \nonumber \\ &=  \sum_{k=1}^\infty \frac{(i z)^k}{k !}\int_0^\infty  e^{-\theta x} x^{k-1} E_{\alpha} \left( -\lambda \, x^{\alpha} \right) dx 
\nonumber \\ &=
    \sum_{k=1}^\infty \frac{(i z)^k}{k !}\sum_{j=0}^\infty \left(-\lambda \right)^j \frac{1}{\Gamma(1+ \alpha j)}\int_0^\infty  e^{-\theta x} x^{\alpha j +	k-1} dx \nonumber \\ &=   \sum_{k=1}^\infty \frac{\left(\frac{i z}{\theta}\right)^k}{k !}\sum_{j=0}^\infty \left(\frac{-\lambda}{\theta^\alpha} \right)^j \frac{\Gamma( k+ \alpha j )}{\Gamma(1+ \alpha j)}  \nonumber \\ &=     \sum_{j=0}^\infty \left(\frac{-\lambda}{\theta^\alpha} \right)^j \frac{1}{\Gamma(1+ \alpha j)} \sum_{k=1}\frac{\left(\frac{i z}{\theta}\right)^k}{k !} \Gamma( k+ \alpha j ) . 
\end{align}
For $j=0$ the summations on $k$ in \eqref{eq:CEgmain}  reduce to a logarithmic series. 
 For $j \geq1$, recalling the  binomial series it holds
\begin{align}\label{eq:bin}\sum_{k=1}\frac{\left(\frac{i z}{\theta}\right)^k}{k !}\frac{ \Gamma( k+ \alpha j )}{\Gamma(\alpha j)}
=\left(\frac{\theta}{\theta-i z}\right)^{\alpha j}-1.
\end{align}
Substitute \eqref{eq:bin} in \eqref{eq:CEgmain}, and  under the convergence conditions $|z| < \theta, \theta^\alpha<\lambda$, we obtain
\begin{align}\label{eq:CEg0}
 & \sum_{j=1}^\infty \left(\frac{-\lambda}{ (\theta-iz)^\alpha} \right)^j \frac{1}{\alpha j} -\sum_{j=1}^\infty \left(\frac{-\lambda}{\theta^\alpha} \right)^j \frac{1}{\alpha j}+ \sum_{k=1}\frac{(i z/\theta)^k}{k}   \nonumber \\&=-\frac{1}{\alpha} \left( \log \left(  \frac{(\theta-i z)^\alpha+\lambda}{(\theta- i z)^\alpha}\right) - \log\left( \frac{\theta^\alpha+\lambda}{\theta^\alpha} \right) + \log\left(\frac{(\theta - i z)^\alpha}{\theta^\alpha} \right) \right) \nonumber \\ &= \frac{1}{\alpha} \log \left(\frac{\theta^\alpha +\lambda}{(\theta-i z)^\alpha+\lambda} \right).
\end{align}
 Here we consider the principal branch of the complex logarithm and the power function $y \mapsto y^\alpha$ for Arg$(y) \in (-\pi, \pi]$.

\medskip

 For $\gamma \in (0,2)$, we observe  that since  $m$ decays exponentially, it holds $\int_{ \{|x| >1\}} x m(x)dx=\mu_1 < \infty$.  We can thus use the representation of the characteristic exponent with constant truncation function 1.
 The same calculations of \eqref{eq:CEgmain} produce
\begin{align}\label{eq:CEgmain2}
 \int_0^\infty & \left(e^{i z  x}-1  - i z x \right) q(x)x^{-1-\gamma} dx=    \theta^\gamma  \sum_{j=0}^\infty \left(\frac{-\lambda}{\theta^\alpha} \right)^j \frac{1}{\Gamma(1+ \alpha j)} \sum_{k=2}\frac{\left(\frac{i z}{\theta}\right)^k}{k !} \Gamma( k-\gamma+ \alpha j ).
\end{align}
Now first let  $\gamma \neq 1$. Using the binomial series again we have
\begin{equation}\label{powerL2}
 \sum_{k=2}^\infty \Gamma(k+ j \alpha- \gamma) \frac{(i z/\theta)^k}{k!}  = \Gamma(j \alpha-\gamma)\left( \left(\frac{\theta - iz}{\theta}\right)^{-\alpha j +\gamma} -  i z \frac{j \alpha -\gamma}{\theta}  	- 1\right) 
\end{equation}
and therefore, whenever $|z|<\theta$ and $\theta^\alpha<\lambda$, \eqref{eq:CEgmain2}   
becomes, in view of the series expression \eqref{eq:Dots} for the $_2R_1$   function 
\begin{align}\label{eq:CEgl2}
&   \sum_{j=0} ^\infty   \frac{\Gamma(j \alpha-\gamma) }{\Gamma(j \alpha +1) } \left((\theta-iz)^{\gamma} \left( \frac{-\lambda }{(\theta- iz)^\alpha} \right)^j - \theta^{\gamma} \left( \frac{-\lambda }{\theta^{\alpha}} \right)^j  \left( iz\frac{j \alpha-\gamma}{\theta}   + 1 \right) \right) \nonumber \\&=   
(\theta-iz)^{\gamma}\sum_{j=0} ^\infty   \frac{\Gamma(j \alpha-\gamma) }{\Gamma(j \alpha +1) } \left( \frac{-\lambda }{(\theta- i z)^\alpha} \right)^j   - \theta^{\gamma}\sum_{j=0} ^\infty   \frac{\Gamma(j \alpha-\gamma) }{\Gamma(j \alpha +1) } \left( \frac{-\lambda }{\theta^\alpha} \right)^j 
\nonumber \\  &\phantom{xxxxxxxxxxxxxxxxxxxxxxxxxxxxxxxx}  -  i z \theta^{\gamma-1}\sum_{j=0} ^\infty   \frac{\Gamma(j \alpha-\gamma+1) }{\Gamma(j \alpha +1 ) } \left( \frac{-\lambda }{\theta^\alpha} \right)^j 
\nonumber \\ & =  \Gamma(-\gamma)\left(  (\theta - iz )^\gamma \, _2R_1\left(1, -\gamma, 1, \alpha;\frac{-\lambda}{ (\theta-i z)^\alpha}\right)  + i z \gamma \, \theta^{\gamma-1} \,  _2R_1\left(1, 1-\gamma, 1, \alpha;-\frac{\lambda}{ \theta^\alpha} \right) \nonumber \right. \\ & \phantom{xxxxxxxxxxxxxxxxxxxxxxxxxxxxxxxx}  \left. -\theta^\gamma \,  _2R_1\left(1, -\gamma, 1, \alpha;-\frac{\lambda}{ \theta^\alpha} \right) \right) .
\end{align}

Finally, for $\gamma=1$ the term $j=0$ in
\eqref{eq:CEgmain2} becomes yet another log-series
\begin{align}\label{powerFlog}
& \sum_{k=2}^\infty \Gamma(k-1) \frac{(i z/\theta)^k}{k!}  = \sum_{k=2}^\infty \frac{(i z/\theta)^k}{k(k-1)} = \sum_{k=2}^\infty \frac{(i z/\theta)^k}{k-1} -\sum_{k=2}^\infty \frac{(i z/\theta)^k}{k} \nonumber 
 \\ &  = \frac{iz}{\theta}+\log\left(1-\frac{iz}{\theta} \right)\left(1- \frac{iz}{\theta} \right).
\end{align}
Separating this term and using   \eqref{powerL2} results in the following expression for \eqref{eq:CEgmain2}
\begin{align}\label{eq:ceplusgamma1} 
&\frac{\theta-iz}{\alpha}\sum_{j=1} ^\infty    \frac{1 }{ j(\alpha j-1) } \left( \frac{-\lambda }{(\theta- iz)^\alpha} \right)^j   - \frac{\theta}{\alpha}\sum_{j=1} ^\infty   \frac{1 }{ j(\alpha j-1) } \left( \frac{-\lambda }{\theta^\alpha} \right)^j 
\nonumber \\  &  -  \frac{iz}{\alpha}  \sum_{j=1} ^\infty  \frac{1}{j}  \left( \frac{-\lambda }{\theta^\alpha} \right)^j +i z +\log\left(1-\frac{iz}{\theta} \right)
\left(\theta- i z \right). \end{align}
Observe that  for $y \in \mathbb  C$, $|y|<1$ it holds that
\begin{align}\label{eq:lerche}
&\sum_{j=1} ^\infty    \frac{y^j }{ j(\alpha j-1) }=\alpha \sum_{j=1} ^\infty    \frac{y^j }{(\alpha j-1) }-\sum_{j=1} ^\infty    \frac{y^j }{j } =  \sum_{j=0} ^\infty    \frac{y^{j+1} }{( j+\frac{\alpha-1}{\alpha}) } + \log(1-y) \nonumber \\ =& y \, \Phi \left(y,1,\frac{\alpha-1}{\alpha}\right)+ \log(1-y). \end{align}
Replacing \eqref{eq:lerche} in \eqref{eq:ceplusgamma1}, and recalling  \eqref{eq:Lerch}, the case $\gamma=1$ specializes to
\begin{align}\label{eq:CEg1} 
  \int_0^\infty & \left(e^{i z  x}-1  - i z x \right) q(x)x^{-\gamma-1} dx= \nonumber \\  &\frac{\lambda}{\alpha} \left(     \frac{\Phi\left(-\frac{\lambda}{\theta^{\alpha}}, 1, \frac{\alpha-1}{\alpha}\right) }{\theta^{\alpha-1}}- \frac{\Phi\left(-\frac{\lambda}{(\theta-iz )^{\alpha}}, 1, \frac{\alpha-1}{\alpha}\right) }{(\theta-iz)^{\alpha-1}} \right) + \frac{\theta-i z}{\alpha} \log\left(\frac{(\theta-i z)^\alpha+ \lambda}{(\theta-i z)^\alpha}\right) -\nonumber \\& \frac{\theta}{\alpha}\log\left(\frac{\theta^\alpha+ \lambda}{\theta^\alpha}\right)  + iz+ \frac{\theta-i z}{\alpha}\log\left(\frac{(\theta -iz)^\alpha}{\theta^\alpha}\right)
 + \frac{i z}{\alpha}\log \left(\frac{\theta^\alpha+\lambda}{\theta^\alpha}\right) \nonumber \\ =&
\frac{\lambda}{\alpha} \left(    \frac{\Phi\left(-\frac{\lambda}{\theta^{\alpha}}, 1, \frac{\alpha-1}{\alpha}\right) }{\theta^{\alpha-1}} -  \frac{\Phi\left(-\frac{\lambda}{(\theta-iz)^{\alpha}}, 1, \frac{\alpha-1}{\alpha}\right) }{(\theta-iz)^{\alpha-1}} \right) + \frac{\theta-i z}{\alpha} \log\left(\frac{(\theta-i z)^\alpha+ \lambda}{\theta^\alpha+ \lambda }\right)+i z .  \end{align}

 Integrating the obtained expressions \eqref{eq:CEg0}, \eqref{eq:CEgl2} and \eqref{eq:CEg1} in $\sigma(du)$ as in \eqref{eq:onedimlev}, subtracting $i z (\mu_0-\mu)$  to  \eqref{eq:CEg0} and adding $iz (\mu_1+\mu)$  to \eqref{eq:CEg1}-\eqref{eq:CEgl2},   $(i)-(iii)$ follow for all $|z|<\theta$ and $\lambda < \theta^{\alpha}$.

\smallskip

 We must finally prove the analyticity on $S$ and that the constraint $\lambda <\theta^\alpha$ can be lifted.
As well-known by \citet{luk+sza:52}, characteristic functions can be analytically continued on horizontal strips $L=\{z \in \mathbb C:  -a< \mbox{Im}(z) < b \}$, $a,b \in \mathbb R \cup \{\infty\} $ in the complex plane.

Let us first treat the case $\gamma=0$.  The function  $y \mapsto y^\alpha$ is analytic outside the logarithm branch cut $\mathbb C \setminus (-\infty, 0]$ yielding the condition $(\theta_+-iz) \notin (-\infty, 0]$, i.e. Im$(z) \neq \theta_+$. Hence, $S^+:=\{z \in \mathbb C: \mbox{Im}(z)<\theta_+  \}$. Furthermore, the logarithm branch cut is never crossed by the function $\frac{\theta_+^{\alpha_+}+\lambda_+}{(\theta_+-iz)^\alpha+\lambda_+}$, so that  $\log\left(\frac{\theta_+^{\alpha_+}+\lambda_+}{(\theta_+-iz)^{\alpha_+}+\lambda_+}\right)$ is analytic on $S^+$. Since a characteristic function must be defined in a neighborhood of the origin, and so is expression $(i)$ by the previous part, by the principle of identity of analytic functions we must have  $b=\theta_+$. A similar reasoning applies to $(\theta_-+iz)$, which is analytical on $S^-:=\{z \in \mathbb C: \mbox{Im}(z)>-\theta_-  \}$, so that $a=\theta_-$, and $S=S^+ \cup S^-$.

For $\gamma \in (0,2)$, $\gamma \neq 1$ , we exploit the fact that it is known (e.g. \citet{kar+pri:19}) that $_2 R_1$ can be analytically continued on $\mathbb C \setminus [1, \infty)$.  On the other hand, on the regions $S^{\pm}$ the functions $-\lambda_\pm/(\theta_\pm \pm iz)^{\alpha_\pm}$ never attain real positive values, and together with the arguments above this shows that $\Psi_X(z)$ is analytic on $S$.

Lerch's transcendent can also be analytically continued on $\mathbb C \setminus [1,\infty)$, e.g. \citet{gui+son:08}, Lemma 2.2, and then the claim  follows also for the case $\gamma=1$. That the condition $\lambda < \theta^\alpha$ can be lifted also follows, since the whole parameter ranges of $\lambda, \theta, \alpha$ always lie in the domain of analyticity of the continued functions.

\end{proof}

\begin{oss}\label{oss:altchar} For $\gamma_+,\gamma_- \in (0,1)$ another expression for the characteristic exponent is available. For such parameter range we still have that as in the $\gamma=0$ case,  $\int_0^\infty  \left(e^{iz x}-1   \right) q(x)x^{-1-\gamma} dx< \infty$ and then the summation in  \eqref{eq:CEgmain2} would yield the expression
\begin{align}\label{eq:CFTGSgl2bis}\psi_X(z)=&\Gamma(-\gamma_+) \delta_+ \left( (\theta_+ - i z )^{\gamma_+} \, _2R_1\left(1, -\gamma_+, 1, \alpha_+;\frac{-\lambda_+}{ (\theta_+-i z)^{\alpha_+}}\right)  \right. \nonumber  \\ & \left.  -\theta_+^{\gamma_+} \,  _2R_1\left(1, -\gamma_+, 1, \alpha_+;-\frac{\lambda_+}{ \theta_+^{\alpha_+}} \right) \right)+ \nonumber \\ &\Gamma(-\gamma_-) \delta_- \left( (\theta_- + iz )^{\gamma_-} \, _2R_1\left(1, -\gamma_-, 1, \alpha_-;\frac{-\lambda_-}{ (\theta_- + i z)^{\alpha_-}}\right)  \right. \nonumber  \\ & \left.  -\theta_-^{\gamma_-} \,  _2R_1\left(1, -\gamma_-, 1, \alpha_-;-\frac{\lambda_-}{ \theta_-^{\alpha_-}} \right) \right) +  i z (\mu-\mu_0 ).
   \end{align}
Assume $X \sim$ GTSG$_\bgamma(\balpha, \blambda, \btheta, \bdelta; \mu_0)$. Equating  \eqref{eq:CFTGSgl2bis} and \eqref{eq:CFTGSgl2}, computing the expression in $z=-i$ and noticing that by \citet{sat:99}, Chapter 25, in this case it holds $E[X]=\mu_0 +\mu_1$ one deduces 
\begin{equation}
E[X]=\frac{\delta_+\Gamma(1-\gamma_+) }{\theta_+^{  1 -\gamma_+}} \, _2R_1\left(1, 1-\gamma_+  , 1, \alpha_+, - \frac{ \lambda_+ }{ \theta_+^{\alpha_+}}  \right)- \frac{\delta_-\Gamma(1-\gamma_-) }{\theta_-^{  1 -\gamma_-}} \,  _2R_1 \left(1, 1-\gamma_-  , 1, \alpha_- ; -\frac{ \lambda_- }{ \theta_-^{\alpha_-}} \right)
\end{equation}
 whenever $\theta_+, \theta_->0$.  We will study cumulants more in detail in Subsection \ref{sec:cum}. 
\end{oss}

\bigskip

\begin{esmp}\label{ex:lintemp} \emph{The} TPL \emph{distribution.}  Consider a  TGS$^+_{\alpha}(\lambda, \theta, \delta; \mu_0  )$
 r.v. $X$. From \eqref{eq:CFTGSg0}  we have the characteristic function 
\begin{equation}
\psi_X(z)= \frac{\delta}{\alpha} \log \left(\frac{\theta^\alpha +\lambda}{(\theta-i z)^\alpha+\lambda} \right)
\end{equation}
 for some positive constants $\lambda, \alpha, \theta$. Letting $c_1=\delta/\alpha$, $c_2=1/(\theta^\alpha+\lambda)$ we can write 
\begin{equation}
\psi_X(z)= -c_1 \log \left( 1+ c_2((\theta-i z)^\alpha- \theta^\alpha) \right)
\end{equation}
which is the characteristic exponent of a TPL$(\alpha, c_2, \theta, c_1)$ r.v.  in the parametrization of \citet{tor+al:21}, confirming Remark \ref{special}, (vi).
\end{esmp}

As a consequence of the analyticity of the characteristic functions above we know by the general theory of \citet{luk+sza:52}, that under the assumptions above a moment generating function for a TGS can be defined, all the moments exist and so do the $\theta$-exponential moments for $\theta \in (-\theta_-,\theta_+ )$.


\begin{oss}\label{rem:scalesum}\emph{ Scaling and independent sums.} Let $c>0$.  If $X$ is as in Theorem \ref{thm:CF},  by inspection on all cases we observe 
\begin{equation}c X \sim   
\mbox{GTGS}_{\bgamma}(\balpha,  (\lambda_+ c^{- \alpha_+}, \lambda_- c^{-\alpha_- }), \btheta c^{-1},( \delta_+ c^{\gamma_+}, \delta_- c^{\gamma_-}); \mu c^{-1}).
\end{equation}
Furthermore if  $X \sim $ \GTGS and $X' \sim $ GTGS$_{\bgamma}(\balpha, \blambda, \btheta,  \bdelta'; \mu')$ are independent and $\btheta>0$, then $X+X' \sim $ GTGS$_{\bgamma}(\balpha, \blambda, \btheta,  \bdelta+ \bdelta';\mu+\mu')$ as it follows directly from the i.d. property. 
\end{oss}

\bigskip

\subsection{Pure Mittag-Leffler tempering}

When the tempering function is a pure Mittag-Leffler function, (that is, $\btheta=0$), we have the \GTGSo class. Distributions in this class are structurally more similar to geometric stable laws, in that their L\'evy measure takes the form of the ratio of a Mittag-Leffler function over a power function. 

Before stating the results we show the following technical lemma on some analytical properties of the $_2R_1$ functions of interest.

\begin{lem}\label{lem:limits}
For   $b, c>0$, $a \leq  b < 1$, and $z \in \mathbb C,$  {\upshape Arg}$(z) \neq \pi/b$, it holds
\begin{enumerate} 
\item[(i)] \begin{equation}\label{eq:lemlimits}
\lim_{z \rightarrow 0 } \frac{ \Gamma(a)}{z^a} \, _2 R_1\left(1,a,1,b; \frac{-c}{z^{b}}\right)=c^{-a/b}\frac{\pi}{b \Gamma(1-a) \sin\left(   \frac{\pi a}{b} \right)},
\end{equation}
\item[(ii)] \begin{equation} \label{eq:lemdiff} \frac{d^n}{d z^n} \frac{\Gamma(a)}{z^a} \, _2 R_1\left(1,a,1,b; \frac{-c}{z^{b}}\right)= (-1)^n \frac{\Gamma(a+n)}{z^{a+n}} \, _2R_1\left(1,n+a,1,b; \frac{-c}{z^{b}}\right).
\end{equation}
\end{enumerate}
\end{lem}

\begin{proof}
 Under the given assumptions we can use  \citet{kil+sa:03}, Corollary 5.2.1 and conclude that for all $w \in \mathbb C$ such that $|\mbox{Arg}(-w) <\pi|$ the following asymptotics when $w \rightarrow \infty$  for such function hold true\footnote{We believe that there is a typo in (5.2) in Theorem 5.2 of \citet{kil+sa:03}. A factor $\mu/\omega$ appears to be missing from the second term, since $A_2=\omega/\mu$ is apparently not accounted for at the denominator of the outer summation, as  it should follow from equation (4.9) of Theorem 4.2 from which equation (5.2) is derived. In our case $\mu=1$, $\omega=b$.}
\begin{align}\label{eq:2R1asympt}
&_2R_1(1, a, 1 ; b, z)=\frac{\Gamma(a-b)}{\Gamma(a)\Gamma(1-b)}(-z^{-1})+O(z^{-1})  \nonumber \\ &+\frac{\Gamma( 1-a/b) \Gamma( a/b)}{b \Gamma(a)\Gamma(1-a)} (-z^{-\frac{a}{b}})+ O(-z^{-\frac{a}{b}}).
\end{align}
Now setting $w=-c/z^b$, since Arg$(z) \neq \pi/b$ then Arg$(-w) \neq \pi$, and as  $z \rightarrow 0$, we have $w \rightarrow \infty$. Thus

\begin{align}
&\Gamma(a) z^{-a}\,_2R_1\left(1, a, 1, b, \frac{-c}{z^b} \right) \sim  z^{-a} \frac{\Gamma(1-a/b) \Gamma( a/b)}{b \Gamma(1-a)}\left( \frac{z^b}{c}\right)^{\frac{a}{b}}  
\end{align}
and the claim follows  from the above and  Euler's reflection formula. 

\medskip

Regarding $(ii)$ first assume $n=1$. We have, using the chain rule and differentiating term by term the uniformly convergent series
\begin{align}\label{eq:diff2r1}
\frac{d}{d z} & z^{-a} \Gamma(a) \, _2 R_1\left(1,a,1,b; \frac{-c}{z^{b}}\right)= -a \Gamma(a) z^{-a-1}  \, _2 R_1\left(1,a,1,b; \frac{-c}{z^{b}}\right) \nonumber \\& - z^{-a} \sum_{j=1}^{\infty}  \frac{\Gamma(b j+ a)}{\Gamma(b j+1)}\frac{(-c)^j}{z^{bj+1}} bj=  -a \Gamma(a) z^{-a-1} \,_2 R_1\left(1,a,1,b; \frac{-c}{z^{b}}\right) \nonumber \\ &+ z^{-a-1}\left(- \sum_{j=1}^{\infty}  \frac{\Gamma(b j+ a+1)}{\Gamma(b j+1)}\frac{(-c)^j}{z^{bj}} + a   \sum_{j=1}^{\infty}  \frac{\Gamma(b j+ a)}{\Gamma(b j+1)}\frac{(-c)^j}{z^{bj}} \right) \nonumber \\& = - z^{-a-1}\left( a \Gamma(a)+ \sum_{j=1}^{\infty}  \frac{\Gamma(b j+ a+1)}{\Gamma(b j+1)}\frac{(-c)^j}{z^{bj}}  \right)=-z^{-a-1}\Gamma(a+1) \, _2 R_1\left(1,a+1,1,b; \frac{-c}{z^{b}}\right).
\end{align}
Using  \eqref{eq:diff2r1} with $a'=a+n$ and appealing to the principle of mathematical induction the proof is complete.

\end{proof}

We can now prove the main result regarding the characteristic exponent of GTGS$^0$ r.v.s.

\begin{thm}\label{thm:CF0}   Let $X$ be a \GTGSo r.v., with $\alpha_+, \alpha_- \in (0,1)$, $\gamma_+, \gamma_- \neq 1$, and define, for $z \in \mathbb R$

\begin{equation}
\Theta_\pm(z)=\cos \frac{\alpha_\pm \pi}{2}\left(1- i\tan \frac{ \alpha_\pm\pi}{2}\mbox{\upshape{sgn}} ( z)   \right).
\end{equation}
We have, in the notation of Theorem \ref{thm:CF}:
\begin{itemize}
\item[(i)] if $ \gamma_+=\gamma_-=0$ then
\begin{equation}\label{eq:CFtheta0gamma0}
 \psi_X(z)= \frac{\delta_+}{\alpha_+}  \log \left(\frac{\lambda_+}{ |z|^{\alpha_+}\Theta_+(z) +\lambda_+} \right)+ \frac{\delta_-}{\alpha_-}  \log \left(\frac{\lambda_-}{ |z|^{\alpha_-}\Theta_-(z) +\lambda_-} \right)  + i  z (\mu-  \mu_0);
\end{equation}
\item[(ii)] if $\max\{\alpha_+ + \gamma_+, \alpha_-+\gamma_-\} \in (0,1)$ then
\begin{align}\label{eq:CFtheta0gammaless1}
& \psi_X(z)=\delta_+ \bigg( \Gamma(-\gamma_+)  |  z| ^{\gamma_+}\Theta_+(z) \,_2R_1\left(1, -\gamma_+, 1, \alpha_+;\frac{-\lambda_+}{ | z |^{\alpha_+}\Theta_+(z)}\right)    - \ell(\gamma_+, \alpha_+, \lambda_+) \bigg)
   \nonumber  \\ +& \delta_-\bigg( \Gamma(-\gamma_-)   | z| ^{\gamma_-}\Theta_-(z) \,_2R_1\left(1, -\gamma_-, 1, \alpha_-;\frac{-\lambda_-}{ | z| ^{\alpha_-}\Theta_-(z)}\right)    - \ell(\gamma_-, \alpha_-, \lambda_-) \bigg)  + i  z (\mu-  \mu_0);\end{align}   
\item[(iii)] if $ \min\{\alpha_+ + \gamma_+, \alpha_-+\gamma_-\} \in [1,3)$ then 
\begin{align}\label{eq:CFtheta0gammaless3}
 \psi_X(z)&=  \delta_+ \bigg( \Gamma(-\gamma_+)  |  z|^{\gamma_+}\Theta_+(z) \,_2R_1\left(1, -\gamma_+, 1, \alpha_+;\frac{-\lambda_+}{ |z  |^{\alpha_+}\Theta_+(z)}\right)    \nonumber  \\ &  - \ell(\gamma_+, \alpha_+,\lambda_+)  - i z \, \ell(\gamma_+-1, \alpha_+,\lambda_+)  \bigg)
  \nonumber \\  +& \delta_- \bigg( \Gamma(-\gamma_-)  |  z|^{\gamma_-}\Theta_-(z) \,_2R_1\left(1, -\gamma_-, 1, \alpha_-;\frac{-\lambda_-}{ |z  |^{\alpha_-}\Theta_-(z)}\right)    \nonumber  \\ &  - \ell(\gamma_-, \alpha_-, \lambda_-)  + i z  \, \ell(\gamma_--1, \alpha_-,\lambda_-)  \bigg)    + i  z (\mu +  \mu_1) ;
\end{align}
\end{itemize}
where  
\begin{equation}\label{eq:ell}
\mathcal \ell(x, y, z)=z^{x/y} \frac{\pi}{\sin\left(-\pi \frac{x}{y} \right)}\frac{1}{y \Gamma(1+x)}. 
\end{equation}
 The cases not accounted by $(i)-(iii)$ can be obtained by combining the expressions for positive and negative parts corresponding to the relevant inequalities satisfied by $\alpha_\pm + \gamma_\pm$.

\end{thm}

\begin{proof} 
 Let $m$ be the L\'evy density of $X$ given by \eqref{eq:onedimlev} and let $\theta_+^n,\theta_-^n $, $n \in \mathbb N$, be two sequences of real numbers such that  $\theta_+^n,\theta_-^n  \rightarrow 0$,  as $n \rightarrow \infty$,  set $\btheta^n=(\theta^n_+, \theta^n_-)$ and let $X_n$ be  GTGS$_\bgamma(\balpha, \blambda,  \btheta^n, \bdelta; 0)$ r.v.s with L\'evy densities $m^n$.

  
We argue by dominated convergence.  The sequence $(e^{i z x}-1-i z x\I_{\{|x|<1 \}})m^n(x)$ is dominated in $x$ for all $z$  by the  integrable function $(e^{i z x}-1-i z x\I_{\{|x|<1 \}})m(x)$.
Therefore  
\begin{align}\label{eq:domCFgen}
& \int_{\mathbb R } (e^{i z x}-1-i z x\I_{\{|x|<1 \}})m(x)dx   = \lim_{n \rightarrow \infty} \int_{\mathbb R}(e^{i z x}-1-i z x\I_{\{|x|<1 \}})m^n(x)  dx. 
\end{align} We consider then the decompositions in positive/negative parts $m(x)=m_+(x)+m_-(x)$ and $m^n(x)=m_+^n(x)+m_-^n(x)$ and for brevity only analyze the positive one, the negative one being identical upon the usual sign and parameter modifications. 

Assume first $\gamma=0$. Using Theorem \ref{thm:CF}, $(i)$, in view of \eqref{eq:domCFgen} and continuity of the logarithm
\begin{align}\label{eq:domCF0} &\int_{\mathbb R } (e^{i z x}-1-i z x\I_{\{|x|<1 \}})m_+(x)dx   = \frac{\delta_+}{\alpha_+}\lim_{n \rightarrow \infty} \log \left(\frac{(\theta_+^n)^{\alpha_+} +\lambda_+}{(\theta_+^n-i z)^{\alpha_+}+\lambda_+} \right) -i \lim_{n \rightarrow \infty} z \mu_0^{n,+}   \nonumber \\ &= \frac{\delta_+}{\alpha_+}\log \left(\frac{ \lambda_+}{(-i z)^{\alpha_+}+\lambda_+} \right) - i  z  \mu^{+}_0 
\end{align}
with  $\mu^+_0=\int_0^1 x \, m_+(x)dx, \mu^{n,+}_0=\int_0^1 x \, m^n_+(x)dx$, when $\mu^{n,+}_0 \rightarrow \mu^+_0$ follows again by dominated convergence. But now, for $\alpha \in (0,1)$, $w \in \mathbb R$ a standard computation (\citet{sat:99}, p. 84--85) shows, \begin{equation}(-i w)^{\alpha}=|w|^\alpha \cos \frac{\alpha \pi}{2}\left(1- i\tan \frac{\pi \alpha}{2}\mbox{\upshape{sgn}} (w) \right) \end{equation}
and, after adding $i z \mu$, \eqref{eq:CFtheta0gamma0} is proved.

Next assume $\alpha_+ +\gamma_+ \in [1,3)$.  
Proceeding as in \eqref{eq:domCF0} we obtain
\begin{align}\label{eq:domCF1} &\int_{\mathbb R} (e^{i z x}-1-i zx\I_{|x| <1})m_+(x)dx  \nonumber \\ &=  \delta_+  \Gamma(-\gamma_+)  | z|^{\gamma_+}\Theta_+(z) \, _2R_1\left(1, -\gamma_+, 1, \alpha_+;\frac{-\lambda_+}{ |z |^{\alpha_+}\Theta_+(z)}\right)  \nonumber \\ &   - \delta_+ \Gamma(-\gamma_+) \lim_{n \rightarrow \infty} (\theta^n_+)^{\gamma} \,   _2R_1\left(1,  -\gamma_+, 1, \alpha_+;-\frac{\lambda_+}{ (\theta^n_+)^{\alpha_+}} \right)   \nonumber \\  - &  i z  \delta_+ \Gamma(1-\gamma_+) \lim_{n \rightarrow \infty}     \, (\theta^n_+)^{\gamma_+-1} \,  _2R_1\left(1, 1-\gamma_+, 1, \alpha_+;-\frac{\lambda_+}{ (\theta^n_+)^{\alpha_+}} \right)   + i  z \mu^{n,+}_1  \end{align}
where $\mu^+_1=\int_{\{x>1\}}m_+(x)dx, \mu^{n,+}_1=\int_{\{x>1\}}m^n_+(x)dx$, with $\mu^{n,+}_1 \rightarrow \mu^+_1$ by dominated convergence, which is ensured by the condition $\alpha_++\gamma_+>1$ combined with the asymptotic relation \eqref{eq:MLasymptInf}.
Now for the limits above we apply Lemma \ref{lem:limits}, $(i)$   with $b=\alpha_+$, $z=\theta^n_+$, $c=\lambda_+$ and $a$ respectively equal to $-\gamma_+$, and $ 1-\gamma_+$, to obtain \eqref{eq:CFtheta0gammaless3}.

The case $\alpha_++\gamma_+ \in (0,1]$ is dealt with similarly, but uses instead  expression \eqref{eq:CFTGSgl2bis} for the characteristic functions of  $X_n$.
\end{proof}

We have been unable to treat the cases $\gamma_+=\gamma_-=1$, since we could not derive asymptotic results for $\Phi$ in the domain of interest, but we conjecture a similar expression to hold. For complex numbers $z \in \mathbb C \setminus \mathbb R$ an asymptotic series is given in \citet{fer+al:17}.  The lack of analyticity of the GTGS$^0$ class compared to the analyticity of GTGS puts these two classes in a similar relationship relationship to the one between the S and CTS classes, in that the GTGS class is just an exponentially tempered version of the GTGS$^0$ class. The case $(i)$ of Theorem \ref{thm:CF0}, i.e. the TGS$^0_\balpha(\blambda, \bdelta; \mu_0)$ family can be of interest for applications, and by analogy with the \BG distribution we call these distributions Bilateral Linnik \BL.


\begin{esmp}\label{ex:LinGeostable} \emph{Linnik and geometric stable distributions.} Continuing Example \ref{ex:lintemp}, by letting $X$ be a   TGS$^{0,+}_\alpha(\lambda, \delta; \mu_0)$ r.v., we have the characteristic exponent 
\begin{equation}
\psi_X(z)= \delta \log \left(\frac{\lambda}{|z|^\alpha \Theta_+(z)+\lambda} \right)
\end{equation}
which  is the characteristic exponent of a  PL$(\alpha,\lambda^{-1}, \delta)$ r.v. 
\end{esmp}

\begin{esmp}\label{ex:BLsub}\emph{Subordinated representation of a Bilateral Linnik L\'evy process and geometric stability.}
  The L\'evy measure for general TGS distributions on the real line is provided as a Bochner integral in \citet{koz+sam:99} and is not of the form \eqref{eq:CFtheta0gamma0}.  Assume $X \in $ TGS$^{0,s}_\alpha( \lambda, \delta  ; \mu_0)=:$BL$^{s}(\alpha, \lambda, \delta)$.
In this case we have, observing that   $\Theta_\pm$ are complex conjugates
\begin{equation}\label{eq:GSsym}
\psi_X(z)=-\delta \log \left(|z|^{2 \alpha} \lambda^{-2}+ 2 \lambda^{-1} |z|^\alpha +1   \right).
\end{equation}  
By indicating $G$ a G$(1, \delta)$ law we see that  it is possible to write $\psi_X(z)=\psi_G(\psi_{Y_1+Y_2}(z))$ for independent stable laws $Y_1 \sim S_\alpha(2 \lambda^{-1}, 0; 0)$ and $Y_2 \sim S_{2 \alpha}(\lambda^{-2}, 0; 0)$. 
 \end{esmp}
 
 Example \ref{ex:BLsub} clarifies a crucial difference between exponential tempering and its geometric counterpart, the Mittag-Leffler tempering. Even if a positive geometric stable law (e.g. Pillai's) can be recovered as a limiting $\theta \rightarrow 0$ case of a TPL one, the same does not happen for geometric tempering on the real line. A BL law cannot possibly put in the form \eqref{eq:geomtransf} after taking such limit, since the sum of logarithms does not recombine as necessary. This is in contrast to the CTS situation, for which as tempering goes to 0 we recover a stable law. We recall that the L\'evy measures of GS distributions are not known in closed form.


\subsection{Cumulants}\label{sec:cum}

The difference in the analytic structure of the characteristic functions of GTGS laws depending upon $\btheta=0$ or $\btheta>0$ highlighted in Theorems \ref{thm:CF} and \ref{thm:CF0} is naturally reflected on  cumulants. As we noticed when $\btheta>0$ the GTGS characteristic function is analytical and thus  all the moments exist and can be computed by differentiating the characteristic function. 
However, when $\theta=0$ a crude analysis of the L\'evy measure using  \eqref{eq:MLasymptInf} yields that e.g. as $x  \rightarrow \infty$ then $m(x) \sim x^{-1-\alpha_+-\gamma_+}$, so that not all the L\'evy moments exist. Because of the equivalence of the finiteness of L\'evy and distribution moments (e.g. \citet{sat:99}, Chapter 25), this implies that not all of the \GTGSo cumulants  will be finite. 
 We have the following proposition.


 

\begin{prop}\label{prop:cum} Let $X$ be a \GTGS r.v. with $\theta_+, \theta_->0$, $\gamma \neq 1$.
Then its cumulants $k^X_n$, $n \in  \mathbb N_0$ are given by
\begin{align} 
k^X_1&=\mu_1 +\mu; 
 \label{eq:cumulants1}  \\
k^X_n& =  \frac{\delta_+\Gamma(n-\gamma_+) }{ \theta_+^{n-\gamma_+}}  \, _2R_1\left(1, n-\gamma_+  , 1, \alpha_+; - \frac{ \lambda_+ }{ \theta_+^{\alpha_+}}  \right)+ \nonumber \\ &\phantom{nnnnnnnnnnnnnnnnnn}(-1)^n \frac{\delta_-\Gamma(n-\gamma_-) }{\theta_-^{  n -\gamma_-}} \,  _2R_1 \left(1, n-\gamma_-  , 1, \alpha_- ; -\frac{ \lambda_- }{ \theta_-^{\alpha_-}} \right), \quad n>1.
 \label{eq:cumulantsn}
\end{align}
Let instead $Y$ be a \GTGSo r.v.. Then $Y$ has finite expectation if and only if $\min\{\alpha_ ++\gamma_+, \alpha_-+\gamma_- \}>1$ and finite variance if and only if $\min\{\alpha_ ++\gamma_+, \alpha_-+\gamma_- \}>2$ in which cases
\begin{align} 
E[Y]&=\mu_1 +\mu 
\label{eq:expTS0} \\
\mbox{\upshape{Var}}[Y] &= \delta_+ \ell(\gamma_+-2,\alpha_+, \lambda_+) + \delta_- \ell(\gamma_--2,\alpha_-,\lambda_-)  \label{eq:varTS0}.
\end{align}


\end{prop}

\begin{proof}

 Denote $m(x)=m_+(x)+m_-(x)$ the  positive and negative parts of the L\'evy density. That $k^X_1=\mu+\mu_1$ for all i.d. distribution is well-known (e.g \citet{sat:99}, Example 25.12). 
In the case $\theta_+, \theta_->0$ we apply  that for analytic distributions cumulants   and L\'evy moments coincide for $n>1$  
 (again \citet{sat:99}, Chapter 25). In our case $k^X_n=\int_{\mathbb R} x^n m(x)dx $ , 
 and for $n \geq 1$ and $\theta_+^{\alpha_+} <\lambda_+$ it holds   
\begin{align}\label{cumLAP}
\int_0^\infty x^n m_+( dx) &=   \delta_+  \int_0^\infty  e^{-\theta_+ x} x^{n-\gamma_+-1} E_{\alpha_+} \left( -\lambda_+ \, x^{\alpha_+} \right) dx    \nonumber \\&= \delta_+ \theta_+^{n-\gamma_+}  \sum_{j=0}^\infty \left(\frac{-\lambda_+}{\theta_+^{\alpha_+}} \right)^j \frac{\Gamma(n-\gamma_++ \alpha_+ j )}{\Gamma(1+ \alpha_+ j)}
\end{align}
leading, together with the analogous calculation for $m_-$, to \eqref{eq:cumulants1}-\eqref{eq:cumulantsn}. The case for arbitrary parameters follows by analytic continuation. Because of the analyticity of $\psi_X$, for $\gamma \neq 0$ the conclusion is also immediate from Lemma \ref{lem:limits}, $(i)$ and $(ii)$, since we can  differentiate $\psi$ in $z$ and take the limit $z \rightarrow 0$ to obtain the cumulants.


Regarding $Y$, again because of \citet{sat:99}, Corollary 25.8 the first statement follows from $\int_{\{ |x|>1 \}}x^k m_\pm(x)dx <\infty$ if and only if $k <\alpha_\pm +\gamma_\pm$ , which once again is a consequence of \eqref{eq:MLasymptInf}.  For the variance it is possible to differentiate twice the characteristic exponent  \eqref{eq:CFtheta0gammaless3} and take the limit $z \rightarrow 0$. To this end, since $2-\gamma_+<\alpha_+$  we can apply Lemma \ref{lem:limits}  with $b=\alpha_+$, $a=-\gamma_+$ , $c=\lambda_+$, and $z$ replaced by $-i z$, so that by part $(i)$  it holds
\begin{equation}
\mbox{Var}[Y^+]=\psi''(0)=\frac{\delta_+\Gamma(2-\gamma_+) }{ (-iz)^{2-\gamma_+}}  \, _2R_1\left(1, 2-\gamma_+  , 1, \alpha_+; - \frac{ \lambda_+ }{ (-i z)^{\alpha_+}}  \right)
\end{equation}
and then, using part $(ii)$ 
\begin{align}
k^{Y^+}_2&= \delta_+ \lim_{z \rightarrow 0} \frac{\delta_+\Gamma(2-\gamma_+) }{ (-iz)^{2-\gamma_+}}  \, _2R_1\left(1, 2-\gamma_+  , 1, \alpha_+; - \frac{ \lambda_+ }{ (-i z)^{\alpha_+}}  \right)= \delta_+ \frac{\lambda_+^{\frac{\gamma_+-2}{\alpha_+}}}{\alpha_+}\frac{\pi}{\sin\left( \pi \frac{2-\gamma_+}{\alpha_+}\right)\Gamma(\gamma_+-1)}  \nonumber \\ &= \delta_+ \ell(\gamma_+-2,\alpha_+, \lambda_+).
\end{align}
The same arguments apply to $k^{Y^-}_2$, and the proof is finished. 

\end{proof}
Therefore GTGS$^0$ distributions have the peculiar property of retaining finite variance for some ranges of parameters, but no other higher moment. As mentioned in the introduction this can capture empirical findings on financial data and make this distribution an ideal candidate to model such quantities.

Information about the existence of moments can also be extracted by the Rosi\'nsky measure, which we shall study in the next section.



\begin{esmp}\emph{Cumulants of a {\upshape TGS} distribution}. When $X \sim$ TGS$_\balpha(\blambda, \btheta, \bdelta; \mu_0)$ we have 
\begin{align} 
k^X_n& =  \frac{\delta_+ (n-1)! }{ \theta_+^{n}}  \, _2R_1\left(1, n  , 1, \alpha_+; - \frac{ \lambda_+ }{ \theta_+^{\alpha_+}}  \right)+ \nonumber \\ &\phantom{nnnnnnnnnnnnnnnnnn}(-1)^n \frac{\delta_- (n-1)! }{\theta_-^{  n }} \,  _2R_1 \left(1, n, \alpha_- ; -\frac{ \lambda_- }{ \theta_-^{\alpha_-}} \right)
 \nonumber
 \\ & = \frac{\delta_+ (n-1)! }{ \theta_+^{n}}  \,  \sum_{j=0}^\infty{(\alpha_+ j +1})_{n-1} \left(\frac{-\lambda_+}{\theta_+^{\alpha_+}} \right)^j   + \nonumber \\ &\phantom{nnnnnnnnnnnnnnnnnn}(-1)^n \frac{\delta_- (n-1)! }{\theta_-^{  n }} \, \sum_{j=0}^\infty{(\alpha_- j +1})_{n-1} \left(\frac{-\lambda_-}{\theta_-^{\alpha_-}} \right)^j .
\end{align}
This extends the TPL cumulant analysis of \citet{tor+al:21} Proposition 2.2. One can show along the lines of such result that the TGS cumulants are given by
\begin{equation}\label{eq:equation}
\kappa^X_n=\frac{\delta_+}{\theta_+^n} g_{n-1}\left( \frac{ -\lambda_+ }{ \theta_+^{\gamma_+} }; \alpha_+ \right) +(-1)^n\frac{\delta_-}{\theta_-^n} g_{n-1}\left( \frac{ -\lambda_- }{ \theta_-^{\gamma_-} } ; \alpha_- \right)
\end{equation}
where $g_n(x;c)$ satisfies the recursion
\begin{equation}\label{eq:recurcumr}
g_{n}(x; c)= x c \frac{d}{dx}g_{n-1}(x;c)+ n g_{n-1}(x;c)
\end{equation}
with $c>0$  and
$g_{0}(x;c)=\frac{1}{1-x}$.

\end{esmp}

\section{Spectral representations, limits and absolute continuity}


We analyze more in detail the structure of one-dimensional GTGS distributions in relation to the theory of \citet{ros:07}. We find the spectral and Rosi\'nsky measures of such laws, identify short and long time L\'evy scaling limits and give conditions for absolute continuity with respect to a stable law, as well as with other GTGS distributions. In order to keep in line with the standard theory, for the most part of this section we assume $\gamma_+=\gamma_->0$ and we remove the boldface  throughout to indicate this. Extensions to asymmetric tempering can be easily obtained.

\begin{prop}\label{prop:spectral} A {\upshape GTGS}$_{\gamma}(\balpha, \blambda, \btheta, \bdelta;\mu )$ with $\alpha_+,\alpha_- \in (0,1)$ and $\gamma>0$ admits both a spectral density $s$ and a Rosi\'nsky density $r_\gamma$ given respectively by 
\begin{align}\label{eq:sdensity}
s(x; \balpha, \blambda,\btheta, \bdelta )&=\delta_+\frac{(x-\theta_+)^{\alpha_+-1}}{\pi }\frac{ \sin( \alpha_+ \pi )}{ \lambda_+^{-1} (x-\theta_+)^{2\alpha_+}+ 2 (x-\theta_+)^{\alpha_+} \cos(\alpha_+ \pi)+\lambda_+}\I_{\{x>\theta_+\}}+ \nonumber \\ & \delta_-\frac{(|x|-\theta_-)^{\alpha_--1}}{\pi }\frac{ \sin( \alpha_-\pi )}{ \lambda_-^{-1} (|x|-\theta_-)^{2\alpha_-}+ 2 (|x|-\theta_-)^{\alpha_-} \cos(\alpha_- \pi)+\lambda_-}\I_{\{x<-\theta_-\}}
\end{align}
and, with $1/0:=\infty$
\begin{align}\label{eq:rdensity}
&r_\gamma(x; \balpha, \blambda,\btheta,\bdelta)=\delta_+\frac{ x^{-\gamma+\alpha_+-1}}{\pi }\frac{(1-\theta_+ x)^{\alpha_+-1} \sin( \alpha_+ \pi)}{  \lambda_+^{-1} (1-\theta_+ x )^{2\alpha_+}+ 2(x(1-\theta_+ x))^{\alpha_+} \cos(\alpha_+ \pi)+\lambda_+ x^{2 \alpha_+}}\I_{\{0<x<\theta^{-1}_+\}}  \nonumber \\ &+ \delta_-\frac{ |x|^{-\gamma+\alpha_--1}}{\pi }\frac{(1-\theta_-| x|)^{\alpha_--1} \sin( \alpha_-\pi )}{  \lambda_-^{-1} (1-\theta_- |x| )^{2\alpha_-}+ 2(|x|(1-\theta_- |x|))^{\alpha_-} \cos(\alpha_- \pi)+\lambda_- |x|^{2 \alpha_-}}\I_{\{-\theta^{-1}_-<x<0\}}.
\end{align}

\end{prop}

\begin{proof}
By e.g. \citet{deOl+aL:11},  equation (2.16),  
the density in the Bernstein representation of $E_\alpha(-\cdot^\alpha)$ is
\begin{equation}\label{eq:poldens}
s^E(x)=\frac{x^{\alpha-1}}{\pi }\frac{ \sin(\alpha \pi  )}{ x^{2\alpha}+ 2 x^{\alpha} \cos(\alpha \pi)+1}, \qquad x>0
\end{equation}
which is the p.d.f of the ratio of two independent $\alpha$-stable random variables (see \citet{jam:06}, p. 9). 
After an application of the  Laplace transform rules we see that the tempering function $q(x)=q(x,1)\I_{\{x>0\}}+q(|x|, -1)\I_{\{x <0 \}} $ is given by 
\begin{equation}\label{eq:qbern+}
q(x)=\left\{ \begin{array}{lr}\displaystyle{\int_0^\infty e^{-y x}  \lambda_+^{-1/\alpha_+}s^{E}\big( (y-\theta_+) \lambda_+^{-1/\alpha_+}\big)}dy & x > 0, \\  \\
\displaystyle{\int_0^\infty e^{-y |x|}  \lambda_-^{-1/\alpha_-}s^{E}\big( (y-\theta_-) \lambda_-^{-1/\alpha_-}\big)}
 dy  & x< 0.  \end{array}  \right.
\end{equation}

Substituting \eqref{eq:poldens} in \eqref{eq:qbern+} and using this in \eqref{eq:specmeas} with $\sigma$ as in  
\eqref{eq:sigma0} we obtain  \eqref{eq:sdensity}. Moreover, writing explicitly \eqref{eq:rosmeas}, for $B \in \mathcal B(\mathbb R_+)$ we have, using the integral substitution $y=1/x$
\begin{align}
R(B)&= \delta_+ \int_{\theta_+}^\infty \I_{B}\left(\frac{\mbox{sgn}(x)}{x} \right)   \frac{x^\gamma}{\pi }\frac{ (x-\theta_+)^{\alpha_+-1} \sin(\pi \alpha_+ )}{ \lambda_+^{-1} (x-\theta_+)^{2\alpha_+}+ 2 (x-\theta_+)^{\alpha_+} \cos(\alpha_+ \pi)+\lambda_+} dx \nonumber \\& =\delta_+\int_0^{1/\theta_+}  \I_{B}(y) \frac{ y^{-\gamma-2}}{\pi }\frac{(y^{-1}-\theta_+)^{\alpha_+-1} \sin(\pi \alpha_+ )}{  \lambda_+^{-1} (y^{-1}-\theta_+)^{2\alpha_+}+ 2 (y^{-1}-\theta_+)^{\alpha_+} \cos(\alpha_+ \pi)+\lambda_+} dy
\nonumber \\ 
 &=\delta_+\int_0^{1/\theta_+}  \I_{B}(y) \frac{ y^{-\gamma-\alpha_+-1}}{\pi }\frac{(1-\theta_+ y)^{\alpha_+-1} \sin(\pi \alpha_+ )}{  \lambda_+^{-1} y^{-2 \alpha_+}(1-\theta_+ y )^{2\alpha_+}+ 2 y^{-\alpha_+}(1-\theta_+ y)^{\alpha_+} \cos(\alpha_+ \pi)+\lambda_+} dy \nonumber \\ & = \delta_+ \int_0^{1/\theta_+}  \I_{B}(y) \frac{ y^{-\gamma+\alpha_+-1}}{\pi }\frac{(1-\theta_+ y)^{\alpha_+-1} \sin(\pi \alpha_+ )}{  \lambda_+^{-1} (1-\theta_+ y )^{2\alpha_+}+ 2(y(1-\theta_+ y))^{\alpha_+} \cos(\alpha_+ \pi)+\lambda y^{2 \alpha_+}} dy.
\end{align}
The analogous computation holds for the negative part when $B \in \mathcal B(\mathbb R_-)$, and  
\eqref{eq:rdensity} follows.


\end{proof}

For a GTGS$_\gamma(\balpha, \blambda,\btheta, \bdelta; \mu)$   we denote  GTGS$_\gamma(r_\gamma(x; \balpha, \blambda,\btheta, \bdelta); \mu)$  the parametrization using the Rosi\'nsky  density $r_\gamma$ and a drift $\mu$.

\bigskip

\begin{oss}\label{oss:cumrev} \emph{Moments revisited}.
According to \citet{ros:07}, Proposition 2.7, finiteness of the moments of a TGS law is in some sense equivalent to the finiteness of the moments of the measure $R$. More precisely the $p$-th moment, $p >0$, is always finite for $p < \gamma$; it is finite for $p>\gamma$ if and only if $\int_{\{|x|>1\}}|x|^p R(dx) <\infty$, and for $p=\gamma$ it is finite if and only if $\int_{\{|x|>1\}}|x|^\gamma \log|x| R(dx) <\infty$. Using \eqref{eq:rdensity}  we see that these integrals always converge for any $p$ whenever $\theta_+,\theta >0$, in accordance with Theorem \ref{thm:CF} and Proposition \ref{prop:cum}. Instead if $\theta_+=\theta_-=0$ we have, with $r$ given by \eqref{eq:rdensity},
\begin{align}\label{eq:asymptr}
x^pr_\gamma(x) \sim \left\{\begin{array}{ll} \displaystyle{\delta_+ \frac{  x^{p-\gamma+\alpha_+-1}}{\pi } \frac{ \sin( \alpha_+ \pi)}{  \lambda_+^{-1} + 2x^{\alpha_+} \cos(\alpha_+ \pi)+\lambda_+ x^{2 \alpha_+}} =O(x^{p-\gamma-1-\alpha_+})}, & x \rightarrow \infty \\ \displaystyle{\delta_- \frac{ |x|^{p-\gamma+\alpha_--1}}{\pi } \frac{ \sin( \alpha_- \pi)}{  \lambda_-^{-1} + 2|x|^{\alpha_-} \cos(\alpha_- \pi)+\lambda_- |x|^{2 \alpha_-}} = O(|x|^{p-\gamma-1-\alpha_-}) }, & x \rightarrow -\infty. 
\end{array} \right.
\end{align}
 Therefore for  $p=\gamma$ it is $x^{-\gamma}r_\gamma(x)\log(x) = O(\log(x)x^{-\alpha_+-1})$ as $ x \rightarrow \infty$ and $x^{-\gamma}r_\gamma(x)\log(x) = O(\log(-x)(-x)^{-\alpha_- -1})$, as $ x \rightarrow -\infty$ which both converge, and thus the boundary moment is finite.
Moreover, when $p >\gamma$, the convergence condition is  $\min\{\alpha_+ + \gamma, \alpha_- + \gamma\}>p$,  again consistently with Proposition \ref{prop:cum}.
Furthermore, still by \citet{ros:07}, Proposition 2.7, the condition for the finiteness of the exponential moments (clearly unavailable when $\theta_+=\theta_-=0$) of order  $\beta>0$ is  $R(\{x: |x|>\beta^{-1}\})=0$. From \eqref{eq:rdensity} the latter holds if and only if $\beta \leq \theta_0=\min\{\theta_+, \theta_-\}$. By standard theory (\citet{luk+sza:52}), this implicates the analyticity of the characteristic function of the GTGS law with positive exponential tempering at least in the strip $B=\{z \in \mathbb C \, : |\mbox{Im}(z)|<\theta_0\}$,  in accordance with Theorem \ref{thm:CF}. 
\end{oss}

Spectral measures are useful to understand the short and long time behavior of tempered geometric stable L\'evy processes. 
 For clarity of exposition we further confine our treatment to the case $\alpha_+=\alpha_-$. How to deal with the  case $\alpha_+ \neq \alpha_-$, involving stable limits with different positive and negative stability indices, should be clear.

\begin{prop}\label{prop:time}

 
 Let $X=(X_t)_{t \geq 0}$ be a  {\upshape GTGS}$_\gamma(r_\gamma(x;  (\alpha,\alpha), \blambda, \btheta, \bdelta), \mu)$ with $\gamma \neq 1$ and $\mu=\mu_0$ if $\gamma \in (0,1)$, $\mu=-\mu_1$ if $\gamma \in (1,2)$. Define for all $h>0$  the scaled processes $X^h=(X_{h t})_{t \geq 0}$. 
 We have, as $h \rightarrow 0$
\begin{itemize}
\item[(i)]  $h^{-1/\gamma}X^h \rightarrow Z$ where $Z=(Z_t)_{t \geq 0}$ is a stable L\'evy process such that $Z_1 \sim$  S$_\gamma(\bdelta; \mu^*)$, with $\mu^*=\mu^*_0$ if $\gamma \in (0,1)$ and $\mu^*=-\mu^*_1$ if $\gamma \in (1,2)$, with $\mu^*_1$, $\mu^*_0$ relative to the L\'evy measure of $Z$; 
\end{itemize}
and, as $h \rightarrow \infty$\begin{itemize} 
\item[(ii)] if $\theta_+=\theta_-  =0$ and $\alpha+\gamma \in (0,1) \cup (1,2)$ then  $h^{-(\alpha+\gamma)^{-1}}X_h \rightarrow Z$ where $Z$ is a stable L\'evy process such that $Z_1 \sim  $ {\upshape S}$_{\alpha+\gamma}((\delta_+^*,\delta_-^*); \mu^*)$ with
\begin{equation}\label{eq:dstars}
\delta^*_+=\frac{\delta_+}{\Gamma(1-\alpha)\lambda_+}, \qquad
 \delta^*_-=\frac{\delta_-}{\Gamma(1-\alpha)\lambda_-}; \end{equation}  
\item[(iii)] if $\theta_+, \theta_->0$ or $\alpha+\gamma>2$ that $h^{-1/2}X^h \rightarrow B$ where $B=(B_t)_{t \geq 0}$ is a Gaussian L\'evy process with triplet $(0, \mbox{\upshape{Var}}[X_1] , 0)$. 
\end{itemize}
 The convergences above are in the Skorohod space $D([0, \infty), \mathbb R)$.
\end{prop}

\begin{proof}
Let $r_\gamma$ be the Rosi\'nsky density \eqref{eq:rdensity}. For $(i)$ from \citet{ros:07}, Theorem 3.1, a sufficient condition for the
statement to hold is that \begin{equation}\label{eq:rosshorttime}
\int_{\mathbb R} x^\gamma r_\gamma(x; \balpha, \blambda,\btheta,\bdelta)dx <\infty.
\end{equation}
When $\theta_+,\theta_->0$ the above is  trivially verified since in that case $r_\gamma$ is supported on a bounded set. When $\theta_+=\theta_-=0$ then using \eqref{eq:asymptr} with $p=\gamma$, one has $x^\gamma r_\gamma(x) \sim O(|x|^{-1-\alpha})$, as $x \rightarrow \pm \infty$  so that \eqref{eq:rosshorttime} still holds.  

Now to show $(ii)-(iii)$ we begin by proving the Gaussian limit in $(iii)$ under the assumption $\theta_ + , \theta_->0$. Denote by $\psi_h(z)$ the characteristic exponent of the L\'evy process $X^h$. By \citet{kal:02}, Theorem 15.17, for the claim to hold it is sufficient to show convergence in distribution which we verify on the characteristic exponents. Write the decomposition of $\psi_h$ in spectrally positive and negative parts as $\psi_h=\psi_h^++\psi_h^-$, Now  we can use the integral \eqref{eq:CEgmain2} since $h^{-1} \sim 0$, which yields, after interchanging the summation order
\begin{align}\label{eq:longgauss}\lim_{h \rightarrow \infty}h^{-1/2}\psi^+_h(z )&=\lim_{h \rightarrow \infty}h\psi^+(z h^{-1/2}) \nonumber \\ &=\theta_+^{\gamma} \delta_+ \lim_{h \rightarrow \infty} h \,   \sum_{j=0}^\infty \left(\frac{-\lambda_+}{\theta_+^{\alpha}} \right)^j \frac{1}{\Gamma(1+ \alpha j)} \sum_{k=2}\frac{\left(\frac{i z h^{-1/2}}{\theta_+}\right)^k}{k !} \Gamma( k-\gamma + \alpha j ) \nonumber \\  &= \delta_+ \lim_{h \rightarrow \infty}  h \,  \theta_+^{\gamma_+-2}  \sum_{j=0}^\infty \left(\frac{-\lambda_+}{\theta_+^{\alpha}} \right)^j \frac{ \Gamma( 2-\gamma + \alpha j ) }{\Gamma(1+ \alpha j)}\left( - \frac{z^2 h^{-1}}{2}\right)+o(1) \nonumber \\  &= -\frac{z^2}{2}\delta_+ \theta_+^{\gamma-2} \Gamma(2-\gamma) \, _2R_1\left(1, 2-\gamma, 1;\alpha, -\frac{\lambda_+}{\theta_+^{\alpha}}\right)
.\end{align}
After carrying out the corresponding computation for $\psi^-_h$, recalling Proposition \ref{prop:cum} we notice that in the final expression the factor Var$[X_1]$ appears as multiplying the characteristic exponent of the standard Brownian motion, which establishes the claim.

To prove $(ii)$ and $(iii)$ when $\theta_+=\theta_-=0$ we proceed by setting $\kappa>0$ and  analyze $\psi_h(z h^{-1/k})=h \psi(z h^{-1/\kappa})$. We must expand \eqref{eq:CFtheta0gammaless1} and \eqref{eq:CFtheta0gammaless3} around $h^{-1} \sim 0$; in order to do this we can use \citet{kil+sa:03}, Theorem 5.2., providing a series representation for the $_2 R_1$ function  of complex argument outside the unit circle. Under our parameter specification it holds  (recall also footnote 1) 	
\begin{align}\label{eq:ellseries}
&\Gamma(-\gamma)\, _2 R_1\left(1, -\gamma, 1, \alpha; w \right)= \nonumber \\ &- \sum_{k=0}^{\infty} \frac{\Gamma(-\gamma-\alpha(k+1))}{\Gamma(1-\alpha(k+1))} (w^{-k-1})+  \frac{(-w^{\gamma/\alpha})}{\alpha }\sum_{k=0}^\infty \frac{\Gamma( (k-\gamma)/\alpha) \Gamma(1-(k-\gamma)/\alpha)}{k! \Gamma(1+\gamma-k)}(-1)^k (-w)^{-k/\alpha} \nonumber \\  &= - \sum_{k=0}^{\infty} \frac{\Gamma(-\gamma-\alpha(k+1))}{\Gamma(1-\alpha(k+1))} (w^{-k-1})+ \sum_{k=0}^\infty \frac{(-1)^k}{k!} \ell(\gamma -k , \alpha , -w)
\end{align}
after using Euler's summation and recalling that $\ell(\cdot, \cdot, \cdot)$ is given by \eqref{eq:ell}. Setting $w=\frac{-\lambda_+}{h^{-\alpha/\kappa}|z|^{\alpha}\Theta_+(z)}$ we obtain further
\begin{align}\label{eq:ellseries2}
&\Gamma(-\gamma)h^{-\gamma/\kappa}\, _2 R_1\left(1, -\gamma, 1, \alpha; \frac{-\lambda_+h^{\alpha/\kappa} }{|z|^{\alpha}\Theta_+(z)} \right)\nonumber \\ &= - \sum_{k=0}^{\infty} \frac{\Gamma(-\gamma-\alpha(k+1))}{\Gamma(1-\alpha(k+1))} \left(\frac{ (- iz)^{\alpha}}{-\lambda_+ }\right)^{k+1}h^{-(\gamma+\alpha(k+1))/\kappa}+ \sum_{k=0}^\infty \frac{(-1)^k}{k!} h^{(\gamma-k)/\kappa} \ell\left(\gamma -k , \alpha , \frac{\lambda_+}{(-iz)^{\alpha}} \right).
\end{align}
If $\alpha + \gamma<1$, observing that $\ell\left(\gamma-k , \alpha , \frac{\lambda_+}{(-iz)^{\alpha}} \right)=(-iz)^{-\gamma}\ell(\gamma-k, \alpha, \lambda_+)$ in the positive part of \eqref{eq:CFtheta0gammaless1} with $\mu=-\mu_0$ we obtain,
\begin{align}\label{eq:ellseries3}
&\delta_+\Bigg( \Gamma(-\gamma)|z|^\gamma \Theta_+(z)h^{-\gamma/\kappa}\, _2 R_1\left(1, -\gamma, 1, \alpha; \frac{-\lambda_+h^{\alpha/\kappa} }{|z|^{\alpha}\Theta_+(z)} \right)- \ell(\gamma,\alpha, \lambda_+) \Bigg)  \nonumber \\ &= \delta_+\Bigg( - (-iz)^{\gamma} \sum_{k=0}^{\infty} \frac{\Gamma(-\gamma-\alpha(k+1))}{\Gamma(1-\alpha(k+1))} \left(\frac{ (- iz)^{\alpha}}{-\lambda_+ }\right)^{k+1}h^{-(\gamma+\alpha(k+1))/\kappa} \nonumber \\ & \phantom{xxxxxxxxxxxxxxxxxxxxxxxxxxx}+ \sum_{k=1}^\infty \frac{(-1)^k}{k!} h^{-k/\kappa} \ell\left(\gamma -k , \alpha , \frac{\lambda_+}{(-iz)^{\alpha}} \right) \Bigg). 
\end{align} The  leading order in  \eqref{eq:ellseries3} corresponds to the term $k=0$ of the first series. We thus have
\begin{align}\label{eq:limitstable}
& \lim_{h \rightarrow \infty} h \, \psi^+(h^{-1/\kappa})=  \delta_+  \lim_{h\rightarrow \infty} h \frac{\Gamma(-\gamma-\alpha)}{\Gamma(1-\alpha)\lambda_+ }(-i z)^{\alpha +\gamma} h^{-(\gamma +\alpha)/\kappa}	=  \delta_+ \frac{\Gamma(-\gamma-\alpha)}{\Gamma(1-\alpha)\lambda_+ }(-i z)^{\alpha + \gamma} 
\end{align}
with the last equality holding if and only if $\kappa=\alpha +\gamma$.  Now it is well-known that (e.g. \citet{sat:99}, Lemma 14.11)
\begin{equation}
(-iz)^{\alpha+\gamma}\Gamma(-\alpha-\gamma)=\int_0^{\infty} \frac{(e^{izx}-1)}{x^{1+\gamma+\alpha}}dx  \qquad \alpha+\gamma \in (0,1), \, z \in \mathbb R.
\end{equation} 
which is the characteristic exponent of a spectrally positive $\alpha+\gamma$ stable r.v. with unit $\delta_+$ and $\mu=\mu^*_0$. Substituting in \eqref{eq:limitstable} produces the positive part of the characteristic exponent of $Z$. Repeating the above for $\psi^-_h$  yields $(iii)$ when $\alpha+\gamma \in (0,1)$.

 If instead $\alpha +\gamma \in (1,2)$ we must use
\eqref{eq:CFtheta0gammaless3} with $\mu=-\mu_1$ and in \eqref{eq:ellseries2} we have, from the relation $\ell\left(\gamma-1 , \alpha , \frac{\lambda_+}{(-iz)^{\alpha}} \right)=(-iz)^{1-\gamma}\ell(\gamma-1, \alpha, \lambda_+)$ that
\begin{align}\label{eq:ellseries4}
&\delta_+ \Bigg( \Gamma(-\gamma)|z|^\gamma \Theta_+(z)h^{-\gamma/\kappa}\, _2 R_1\left(1, -\gamma, 1, \alpha; \frac{-\lambda_+h^{\alpha/\kappa} }{|z|^{\alpha}\Theta_+(z)} \right)- \ell(\gamma,\alpha, \lambda_+)- i z h^{-1/\kappa}\ell(\gamma-1,\alpha, \lambda_+) \Bigg) \nonumber \\ &= \delta_+  \Bigg(-  (-iz)^{\gamma} \sum_{k=0}^{\infty} \frac{\Gamma(-\gamma-\alpha(k+1))}{\Gamma(1-\alpha(k+1))} \left(\frac{ (- iz)^{\alpha}}{-\lambda_+ }\right)^{k+1}h^{-(\gamma+\alpha(k+1))/\kappa}+ \nonumber \\& \phantom{xxxxxxxxxxxxxxxxxxxxxxxxxxxxxxxxxxxx} \sum_{k=2}^\infty \frac{(-1)^k}{k!} h^{-k/\kappa} \ell\left(\gamma -k , \alpha , \frac{\lambda_+}{(-iz)^{\alpha}} \right)\Bigg)
\end{align}
Now, when $\alpha + \gamma \in (0,2)$ again the $k=0$ term in the first series leads, and the statement follows as in \eqref{eq:limitstable} this time observing that
\begin{equation}
(-iz)^{\alpha+\gamma} \Gamma(-\alpha-\gamma)=\int_0^{\infty} \frac{(e^{izx}-1-izx)}{x^{1+\gamma+\alpha}}dx  \qquad \alpha+\gamma \in (1,2), \, z \in \mathbb R,
\end{equation}
which is again the characteristic exponent of a spectrally positive $\alpha+\gamma$ stable r.v. and unit $\delta_+$, but this time with $\mu=-\mu^*_1$.
This concludes the proof of $(ii)$. 

Finally if $\alpha +\gamma>2$ in \eqref{eq:ellseries2} then the leading order corresponds to the term $k=2$ in the second series, so that
\begin{align}\label{eq:limitstableg1}
&\lim_{h \rightarrow \infty} h \, \psi^+(h^{-1/\kappa})= \delta_+ \lim_{h \rightarrow \infty} h \frac{h^{-2/\kappa}}{2} (-iz)^{\gamma} \ell\left(\gamma-2, \alpha,\lambda_+ \right)(-iz)^{2-\gamma}=-\frac{ z^2}{2} \delta_+ \ell \left(\gamma-2, \alpha,\lambda_+ \right)  
\end{align}
provided that $\kappa=2$. Together with the analogous computation for $\psi^-_h$ this completes the proof of $(iii)$ and of the Proposition.

 
\end{proof}

Observe that in view of Proposition \ref{prop:cum} (or Remark \ref{oss:cumrev})   the second condition in $(iii)$ of Proposition \ref{prop:time} is equivalent to the finiteness of the variance.
We see that the familiar short time stable/long time Gaussian behavior of CTS laws (e.g. \citet{kuc+tap:13}) is  reproduced for positive $\theta_\pm$ or when the variance of the  \GTGSo process is finite, according to the CLT intuition. Therefore $(iii)$ precisely account for a persistent power law behavior of the model coupled with Gaussian limit, consistently with the estimates in the empirical studies motivating our work. 

Instead, the case $(ii)$ shows that by using Mittag-Leffler tempering  we can also obtain a stable limit whose stability index is increased by the Mittag-Leffler factor $\alpha$, compared to the short time $\gamma$-stable limit. This regime may also be of interest for application.

Another useful information that can be derived from the measure $R$ is the range of values of $\gamma$ for which  absolute continuity with respect to a stable processes holds. The situation is  akin to the generalized exponential tempering situation described in \citet{gra:16}.

\begin{prop}\label{prop:stableequiv}

Let $X=(X_t)_{t \geq 0}$ be a {\upshape GTGS}$_\gamma(r_\gamma(x; \balpha, \blambda, \btheta, \bdelta), \mu)$ L\'evy process. Assume that  there exists  a second probability measure $P'$ under which  $X$ is a  S$_\gamma(\bdelta, \mu^*)$ L\'evy process, with $\sigma$ given by \eqref{eq:sigma0}. 
  Then $P'$ is absolutely continuous with respect to $P$ if and only if
\begin{equation}
\mu^*=\left\{\begin{array}{ll}
\mu &  \mbox{ if } \gamma \in (0,1), \\
\displaystyle{\mu +\int_{\mathbb R_+} x (\log|x|-1)r_\gamma(x; \balpha, \blambda, \btheta, \bdelta) dx} & \mbox{ if } \gamma=1, \\
\displaystyle{\mu+\Gamma(1-\gamma)\int_{\mathbb R} x \,r_\gamma(x; \balpha, \blambda, \btheta, \bdelta) dx} & \mbox{ if } \gamma \in (1,2)
\end{array} \right. 
\end{equation} and  $\min\{\alpha_+, \alpha_-\}>\gamma/2$. Furthermore, in such case there exists a density L\'evy process $Z=(Z_{t})_{t \geq 0}$ such that for all $t >0$
\begin{equation}
\frac{d P'}{d P}\Big|_{\mathcal F_t}=e^{Z_t}.
\end{equation}  
\end{prop}

\begin{proof}
Indicating $q$ the GTGS tempering function in Cartesian coordinates,  from \citet{ros:07}, Theorem 4.1, we have that a necessary and sufficient condition for the  absolute continuity  of a TS$_\gamma$ L\'evy process with  respect to the given stable one is
\begin{equation}\label{abscont}
\int_{\{|x|<1\}}  (1-q(x))^2 x^{-\gamma-1} dx <\infty.\end{equation}
That is, for $x>0$
\begin{equation}
\int_0^1 \left(1-  e^{-\theta_+ x}E_{\alpha_+} ( -\lambda_+ \, x^{\alpha_+}) \right)^2 dx < \infty.
\end{equation}
We have, when $x \sim 0$, using \eqref{eq:MLasympt0} 
\begin{equation}
e^{-\theta_+ x} \sim 1-\theta_+ x , \qquad E_{\alpha_+} ( - \lambda_+ \, x^{\alpha_+}) \sim 1 - \frac{\lambda_+ x^{\alpha_+}}{\Gamma(\alpha_+ +1)}
\end{equation}  and hence
\begin{equation}
1-  e^{-\theta_+ x}E_{\alpha_+} ( -\lambda_+ \, x^{\alpha_+}) \sim  \theta_+ x +  \frac{\lambda_+ x^{\alpha_+}}{\Gamma(\alpha_+ +1)}, \qquad x \sim 0.
\end{equation}
Since $\alpha_+ <1$ this implies the following leading order for $x \sim 0$
\begin{equation}
\left(1-  e^{-\theta_+ x}E_{\alpha_+} ( -\lambda_+ \, x^{\alpha_+}) \right)^2 \sim  \frac{\lambda_+^2  x^{2 \alpha_+ }}{\Gamma(\alpha_+ +1)^2}.
\end{equation}

Comparing with \eqref{abscont} we see that the condition for convergence is $\alpha_+> \gamma /2$. The same calculation on $-1<x<0$ in \eqref{abscont} establishes the claim.  
\end{proof}

When comparing among them two TGS L\'evy processes, conditions of absolute continuity are  somewhat analogous to the TS case (see e.g. \citet{con+tan:03}, Example 9.1), but additional constraints on the Mittag-Leffler parameter are present. Below,  spectral measures do not play a role, so we allow asymmetric stability indices.

\begin{prop}\label{prop:mutualequiv}

Let $X=(X_t)_{t \geq 0}$ be a {\upshape GTGS}$_{\bgamma}(m_{\bgamma}(x; \balpha,\blambda, \btheta, \bdelta), \mu)$ L\'evy process. Assume that  there exists  a second probability measure $P'$ under which  $X$  is a {\upshape GTGS}$_{\bgamma'}(m_{\bgamma'}(x; \balpha',\blambda', \btheta', \bdelta'), \mu')$ L\'evy process. Then $P \sim P'$ if and only if $\bgamma=\bgamma'$, $\bdelta=\bdelta'$ and $\min\{\alpha_\pm,\alpha_\pm'\} >\gamma_\pm/2$. Furthermore, in such case
\begin{equation}\label{eq:martdens}
\frac{d P'}{d P }\Big|_{\mathcal F_t}=e^{Z}
\end{equation}
where $Z=(Z_t)_{t \geq 0}$ is the L\'evy process with triplet $(\mu_Z, 0, m_Z(x)dx)$ given by \begin{align}
\mu_Z=&-\int_{\mathbb R} \left(e^{izx}-1-izx\I_{\{|x|<1\}}\right) m_Z(x)dx \\ m_Z(x)&=m_\bgamma(l^{-1}(x); \balpha,\blambda, \btheta, \bdelta)
\end{align} where

\begin{equation}\label{eq:martdensl}
l(x)=(\theta_+-\theta'_+)x+\log\left( \frac{E_{\alpha'_+} (-\lambda'_+x^{\alpha'_+})}{E_{\alpha_+}(-\lambda_+x^{\alpha_+})} \right)\I_{\{x>0\}}+(\theta'_--\theta_-)x+\log\left( \frac{E_{\alpha'_-} (-\lambda'_-|x|^{\alpha'_-})}{E_{\alpha_-}(-\lambda_-|x|^{\alpha_-})} \right)\I_{\{x<0\}}.
\end{equation}

\end{prop}

\begin{proof}
 According to \citet{sat:99}, Theorem 33.2, $P \sim P'$ if and only if the Hellinger distance between the absolutely continuous L\'evy measures $m(x)dx$ and $m'(x)dx$  is finite, that is 

\begin{equation}
\int_{\mathbb R} \left(\sqrt{m(x)}-\sqrt{m'(x)} \right)^2 dx <\infty.
\end{equation}
Furthermore $l(x)=\log(m'(x)/m(x))$, and then \eqref{eq:martdensl} is clear by  \eqref{eq:onedimlev}, once the first assertion is proved.

Now letting $I(x)=\left(\sqrt{m(x)}-\sqrt{m'(x)} \right)^2$ and using  \eqref{eq:MLasymptInf}, for large $x$   we have
\begin{align}\label{hel0} 
I(x) \sim \left( \sqrt{\frac{\delta_+}{\lambda_+\Gamma(1-\alpha_+) x^{\gamma_+ +\alpha_++1}} }e^{-x\theta_+/2}  - \sqrt{\frac{\delta'_+}{\lambda'_+ \Gamma(1-\alpha'_+) x^{\gamma'_+ +\alpha'_++1}} }e^{-x\theta'_+/2} \right)^2 
\end{align}
which is always integrable at $+\infty$, whatever the value of $\theta_+$ by \eqref{eq:MLasymptInf}. The corresponding convergence holds for large negative $x$.

In a right neighborhood of 0 we have the  condition
\begin{align}\label{hel1}
\int_0^1 &\left( \sqrt{\delta_+ \frac{E_{\alpha_+}(-\lambda_+ x^{\alpha_+})}{x^{1+\gamma_+}}}e^{-x\theta_+/2}    -   \sqrt{\delta'_+ \frac{E_{\alpha'_+}(-\lambda'_+ x^{\alpha'_+})}{x^{1+\gamma'_+}}}e^{-x\theta'_+/2}\right)^2 dx  
< \infty.
\end{align}
We  write
\begin{equation}\label{hel2}I(x) = \delta'_+ \frac{e^{-\theta_+' x}E_{\alpha_+'}(-\lambda_+' x^{\alpha_+'})}{ x^{\alpha_+'+1}}\left(x^{\frac{\gamma_+'-\gamma_+}{2}}  \sqrt{\frac{\delta_+ E_{\alpha_+}(-\lambda_+ x^{\alpha_+})  }{\delta'_+ E_{\alpha_+'}(-\lambda_+' x^{\alpha_+'})} } e^{\frac{-\theta_+-\theta_+'}{2}x}   -1 \right)^2, \quad x>0. \end{equation}

Assuming $\gamma'_+<\gamma_+$ and $\delta_+ \neq \delta_-$, since $\exp(\cdot)\sim E_\alpha(\cdot) \sim 1$ as $x \rightarrow 0^+$, we have $I(x) \sim \delta_+ x^{-1-\alpha_+'-\gamma_+ +\gamma'_+}$, which diverges. The case $\gamma_+ >\gamma'_+$ is similarly excluded. Therefore, for convergence we must have $\gamma_+=\gamma_+'$. Once this is further assumed, if we allow $\delta_+ \neq \delta_+'$ then $I(x)\sim (\sqrt{\delta_+}-\sqrt{\delta_+'})^2x^{-1-\alpha'_+}$, again a divergent integrand. Therefore for convergence it is necessary that both  $\gamma_+=\gamma_+'$ and $\delta_+=\delta_+'$. Expanding the right hand term in  \eqref{hel2} in its McLaurin series using \eqref{eq:MLasympt0} we obtain, with $\alpha_*=\min\{\alpha_+, \alpha_+' \}$
\begin{align}
I(x)\sim&\frac{ \delta_+}{x^{\gamma_++1}}\left( \left( \frac{1- \frac{\lambda_+ x^{\alpha_+}}{\Gamma(1+\alpha_+)}}{1-\frac{\lambda'_+ x^{\alpha'_+}}{\Gamma(1+\alpha'_+)}} \right)^{1/2} -1 \right)^2 \sim \frac{ \delta_+}{x^{\gamma_++1}}\left(      \left(1- \frac{ \lambda_+ x^{\alpha_+}}{2\Gamma(1+\alpha_+)} \right)\left(1+   \frac{\lambda'_+ x^{\alpha'_+}}{2 \Gamma(1+\alpha'_+)} \right) -1 \right)^2 \nonumber \\ & \sim \frac{ \delta_+}{x^{\gamma_++1}} \left( c x^{\alpha_*} \right)^2 = \frac{  c^2 \delta_+}{x^{\gamma_++1 - 2\alpha_*}}
\end{align}
for some $ c \neq 0$. The convergence condition is thus $\gamma_+ +1 - 2\alpha_*<1$, i.e. $\alpha_*>\gamma_+/2$. Repeating in a left neighborhood of 0 proves the result.
\end{proof}
We observe that the condition of equivalence of Proposition \ref{prop:stableequiv} is embedded in Proposition \ref{prop:mutualequiv}. Therefore, in order to  have mutual equivalence between two GTGS$_\gamma$ distributions under different measures, it is necessary that both are equivalent to the same $\gamma$-stable law.

\section{Conclusions}

Motivated by the investigation of power-law tempering in physical and economic systems, we have proposed a general notion of geometric tempering using an exponentially-dampened Mittag-Leffler function. In particular  Propositions  \ref{prop:cum} and \ref{prop:time}  make viable the introduction of processes with Gaussian limit but heavy tails at all time lags in the form of 
GTGS$^0$ laws, entailing (very slow) CLT convergence. Also, when $\theta \neq 0$, familiar exponential/semi-heavy tails are obtained, but  in principle a faster reversion to Gaussian ought to be observable (because of \eqref{eq:MLasympt0}) compared to the classic CTS. 

To  illustrate such effects  we plot in Figure 1 some comparisons  between L\'evy densities/ tempering functions of S/CTS symmetric laws with a GTGS, $\theta >0$, and GTGS$^0$ counterpart.  Around 0 the L\'evy measures are all unbounded, with the stable law showing the fastest blow-up rate. All of the tempered stable laws considered have visually similar asymptotics, and the specific ordering can be worked by comparing and combining the  leading order $e^{-\theta x} \sim 1-\theta x$ with \eqref{eq:MLasympt0}. The tails of the tempered stable L\'evy densities, according to the theory developed thus far, are all lighter than that of the S$^s$ distribution. However the CTS$^s_\gamma$ and GTGS$^{s}_\gamma$ distributions have similar tails, whereas those of the GTGS$^{0,s}_\gamma$  follow a markedly heavier power law, as predicted by our results. Furthermore, since for the chosen parameters  we have $\alpha+\gamma>2$, the variance of the law is finite. In the right panel we visualize the corresponding tempering functions. We notice a stronger tempering around 0 of the GTGS and GTGS$^0$ law compared to CTS, but after the cross-over, as $|x|$ gets larger the tails of GTGS tempering approach those of the exponential, while those of the pure Mittag-Leffler tempering function remain much heavier.

In terms of possible extensions of the GTGS distribution class, we notice that one limitation of Mittag-Leffler tempering is that the tail index of the probability laws is always confined to be smaller than 3, which implicates diverging higher moments, at odds with some of the estimates in financial data (e.g. \citet{gopikrishnan1999scaling}) which instead find evidence of finite skewness. In view of \eqref{eq:MLasymptInf}, this could be resolved  considering geometric tempering with the three-parameter Mittag-Leffler function, which is known to be completely monotone under some parameter constraints (see \citet{gor+al:21}). We leave this direction of research for further investigations.

\begin{figure}
\includegraphics[scale=0.55]{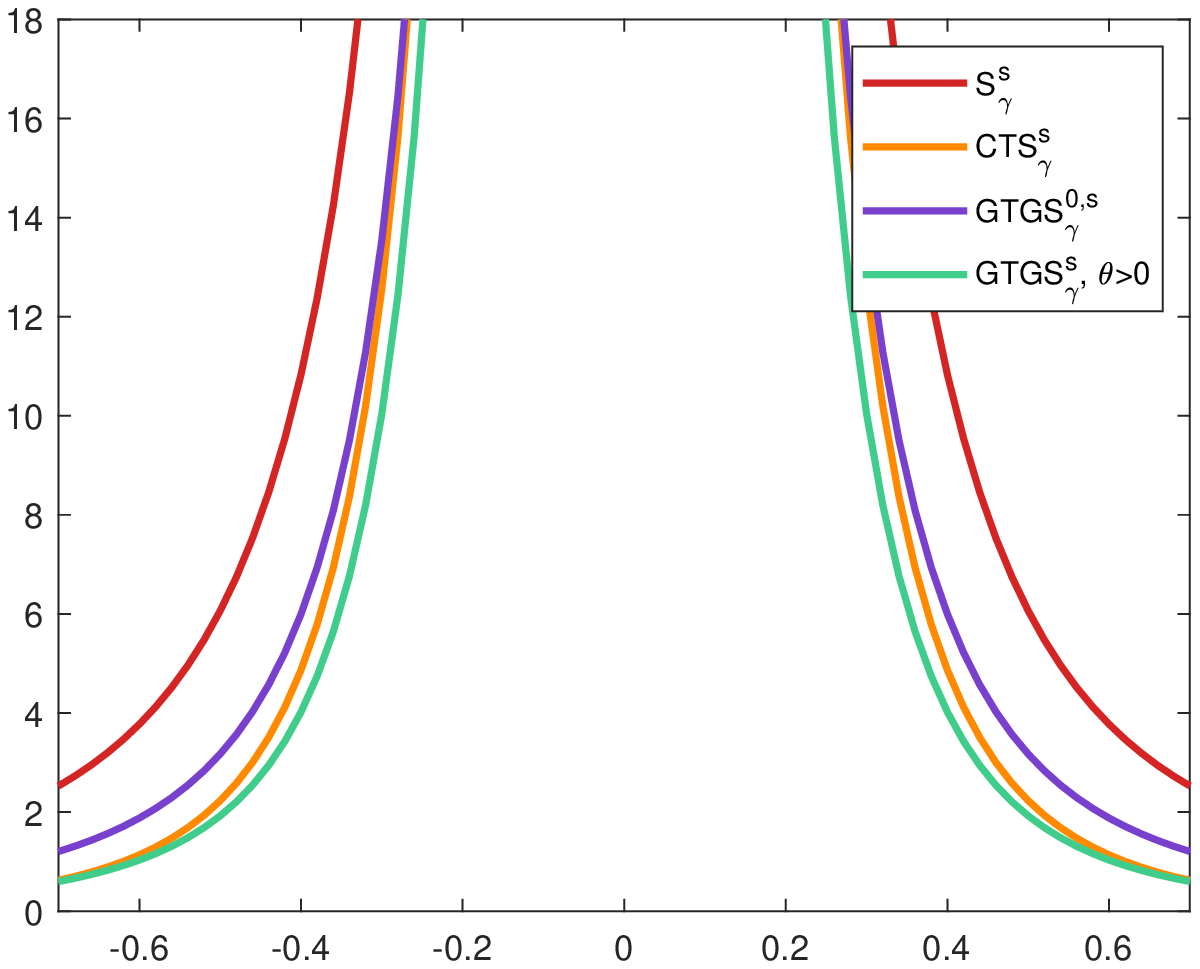}
\includegraphics[scale=0.55]{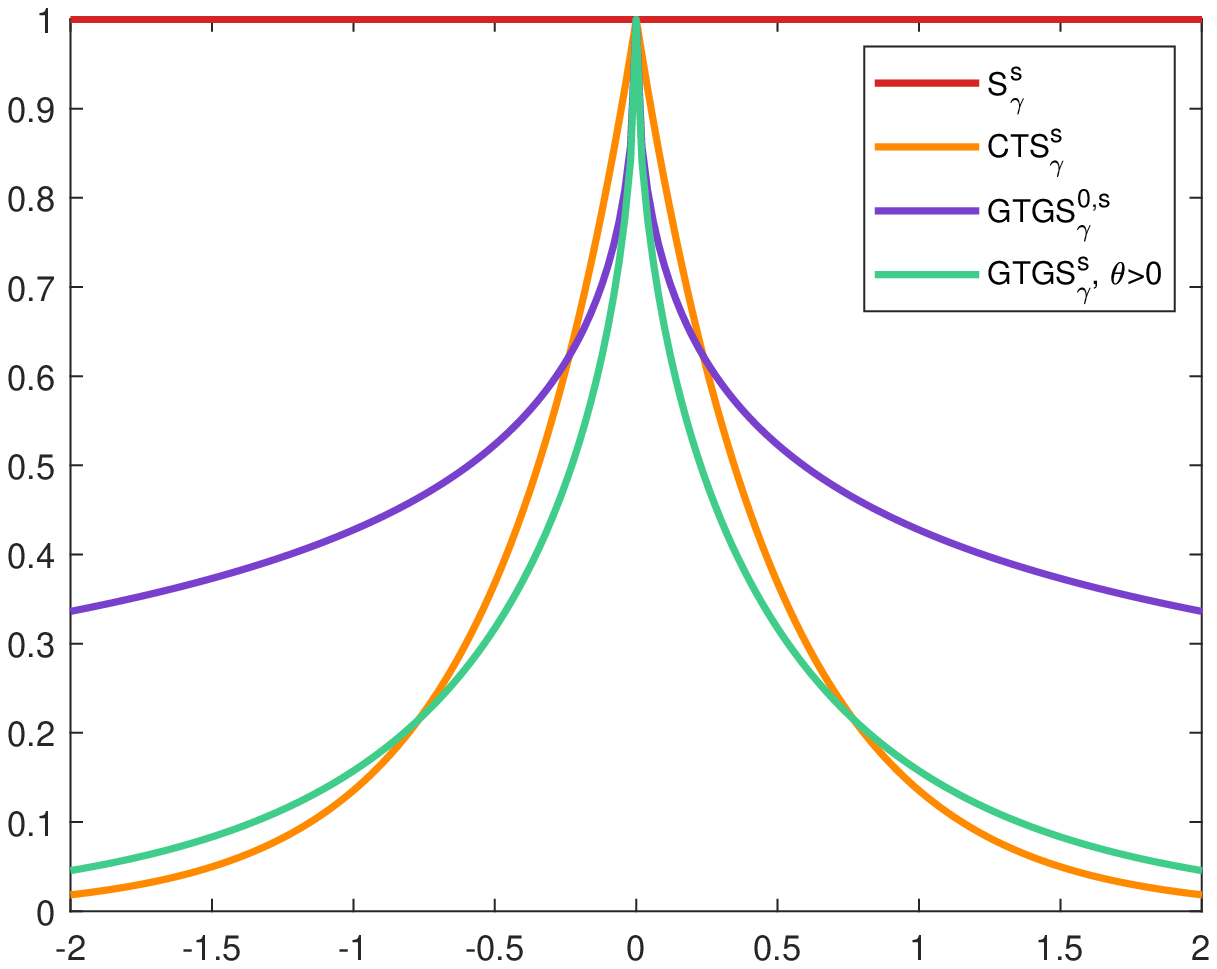}
\caption{\small Left panel:  $S$,   CTS$^s_\gamma$, GTGS$^{s}_\gamma$, and GTGS$^{0,s}_\gamma$  L\'evy densities. Right panel: corresponding tempering functions.  
 The parameters are $ \gamma=1.6, \alpha=0.5, \lambda=\theta=\delta=1 $.}\end{figure}


	\begin{center}
	\begin{table}
		\begin{tabular} {|l|l|c|}
			\hline
			\textbf{Symbol}& \textbf{Description} & \textbf{Dimension}
						\\ \hline
\BG &  bilateral gamma & 1 \\ 
\BL &  bilateral Linnik & 1 \\ 
\CTS & classical  tempered-stable & 1 \\
G$(\lambda,\delta)$ &   gamma & 1 \\
\GTGS & generalized tempered geometric stable & 1  			\\
 GTGS$_{\gamma}(\sigma, \alpha, \lambda, \theta; \mu  )$	& GTGS, spherical parametrization & $d$    \\			GTGS$_\gamma(r_\gamma(x; \balpha, \blambda,\btheta, \bdelta); \mu)$ &  GTGS, Rosi\'nsky parametrization & 1 \\
\GTGSo & purely Mittag-Leffler  tempered GTGS & 1		\\	
\ML & Mittag-Leffler & 1 \\
\PL & positive Linnik & 1\\
 \Sb &			$\alpha$-stable, classical parametrization & 1   \\ \S  &			$\alpha$-stable, L\'evy parametrization & 1  \\
 S$_\alpha(\sigma, \mu)$  & $\alpha$-stable, spherical parametrization & $d$  
\\
\TGS &   tempered  geometric stable & 1 \\	
TGS$(\sigma, \alpha, \lambda, \theta; \mu  )$ &  TGS, spherical parametrization  & $d$  \\ 
\TGSo &  purely Mittag-Leffler tempered TGS  & 1 \\	
\TPL &   					 tempered positive Linnik & 1  \\
TS$_{\gamma}(Q; \mu  )$	& Rosi\'nsky  tempered stable, spectral parametrization & $d$  \\		
TS$_{\gamma}(R; \mu  )$	& Rosi\'nsky  tempered stable, Rosi\'nsky parametrization & $d$  \\
		\hline 
					\end{tabular}
					\caption{\small{Distribution list. When appearing, the superscripts $s, +, -$ stand respectively for the symmetric, spectrally positive and spectrally negative version of the corresponding distribution.}}
					\end{table}
	\end{center}

\section*{Acknowledgments}
Part of this work has been presented at the Third Italian Meeting on Probability and Mathematical Statistics, Bologna, June 2022 and at the FraCalMo Workshop, Bologna, October 2022. The author would like to thank Enrico Scalas, Aleksei Chechkin, Lucio Barabesi, Luisa Beghin, Federico Polito and Piergiacomo Sabino for the helpful comments and discussions.

\section*{Declarations}

 This research did not receive any specific grant from funding agencies in the public, commercial, or not-for-profit sectors. The author declares no competing interests.

\bibliographystyle{apalike}
\bibliography{Bibliography}

\end{document}